\documentclass[reqno,12pt]{amsart}
\newcommand{\CC}{{\mathbb C}}

\def\bege{\begin{equation}} \def\ende{\end{equation}}
\def\begr{\begin{eqnarray}} \def\endr{\end{eqnarray}}
\usepackage{amsmath}
\usepackage{amssymb}
\usepackage{cite}
\usepackage{bbm}
\usepackage{amsmath}
\usepackage{txfonts}
\usepackage{stmaryrd}
\usepackage{amssymb}
\usepackage{amsfonts}
\usepackage{times}
\usepackage{mathrsfs}
\usepackage{amsthm}
\usepackage{enumerate}
\usepackage{framed}
\usepackage{lipsum}
\usepackage{color}

\def\BB{ \mathbb{B}}
\def\SS{ \mathbb{S}}
\def\CC{ \mathbb{C}}

\def\B{\mathcal{B}}
\def\R{\mathcal{R}}
\def\I{\mathcal{I}}
\def\D{\mathbb{D}}
\def\N{\mathbb N}

\def\hD{\hat{\mathcal{D}}}
\def\a{\alpha}

\def\vp{\varphi}
\def\om{\omega}

\def\p{{\prime}}
\def\begr{\begin{eqnarray}} \def\endr{\end{eqnarray}}

\def\msk{\medskip}

\newtheorem{Lemma}{Lemma}
\newtheorem{Theorem}{Theorem}
\newtheorem{Corollary}{Corollary}
\newtheorem{Proposition}{Proposition}
\newtheorem{Remark}{Remark}

\begin{document}
\title[  ]{ Weighted Bergman spaces induced by  doubling weights in the unit ball of $\mathbb{C}^n$ }
 \author{ Juntao Du,    Songxiao Li$\dagger$, Xiaosong Liu and Yecheng Shi  }
 \address{Juntao Du\\ Faculty of Information Technology, Macau University of Science and Technology, Avenida Wai Long, Taipa, Macau.}
 \email{jtdu007@163.com  }
 \address{Songxiao Li\\ Institute of Fundamental and Frontier Sciences, University of Electronic Science and Technology of China,
 610054, Chengdu, Sichuan, P.R. China. } \email{jyulsx@163.com}

\address{Xiaosong Liu \\ S Department of Mathematics, Shantou University, Shantou 515063,  P. R. China }\email{gdxsliu@163.com}

\address{Yecheng Shi\\    School of Mathematics and Statistics, Lingnan Normal University,
     Zhanjiang 524048, Guangdong, P. R. China}\email{ 09ycshi@sina.cn}
 \begin{abstract}  This paper is devoted to the study of the weighted Bergman space $A_\omega^p $ in the unit ball $\mathbb{B}$ of $\mathbb{C}^n$ with doubling weight $\omega$  satisfying
$$\int_r^1\omega(t)dt   <C \int_{\frac{1+r}{2}}^1\omega(t)dt   ,\,\,  0\leq r<1.$$ 
 The $q-$Carleson measures for $A_\omega^p$ are characterized in terms of a neat geometric condition involving Carleson   block.  Some equivalent characterizations for  $A_\omega^p$ are obtained by using the radial derivative and admissible approach regions.   The boundedness and compactness of Volterra integral operator $T_g:A_\omega^p\to A_\omega^q$ are also investigated in this paper with $0<p\leq q<\infty$, where  $$T_gf(z)=\int_0^1 f(tz)\Re g(tz)\frac{dt}{t}, ~~\qquad~~~~f\in H(\mathbb{B}), ~~z\in \mathbb{B}.  $$
   \thanks{$\dagger$ Corresponding author.}
 \end{abstract}
 \thanks{This  project was partially supported by the Macao Science and Technology Development Fund (No.186/2017/A3) and NNSF of China (No. 11720101003).   }
 \maketitle

\section{Introduction}

Let $\BB$ be the open unit ball of $\CC^n$ and $\SS$   the boundary of $\BB$. When $n=1$, then $\BB$ is  the open unit disk in complex plane $\mathbb{C}$ and always denoted by $\D$. Let $H(\BB)$ denote the space of all holomorphic functions on $\BB$. For any two points
$$z=(z_1,z_2,\cdots,z_n)\,\mbox{ and } \, w=(w_1,w_2,\cdots,w_n)$$
in $\CC^n$, we define
$\langle z,w \rangle=z_1\overline{w_1}+\cdots+z_n\overline{w_n}$ and
$$|z|=\sqrt{\langle z,z \rangle}=\sqrt{|z_1|^2+\cdots+|z_n|^2}.$$

 Let $d\sigma$  and $dV$ be the normalized surface and volume measures on $\SS$ and $\BB$, respectively.
For $0<p\leq \infty$, the  Hardy space $H^p(\BB) $(or $H^p$) is the space consisting of all functions $f\in H(\BB)$ such that
$$\|f\|_{H^p}:=\sup_{0<r<1}M_p(r,f),$$
where
$$M_p(r,f)=\left(\int_{\SS}|f(r\xi)|^pd\sigma(\xi)\right)^\frac{1}{p}<\infty, \,\,\mbox{ when }\,\,0<p<\infty,$$
and
$$M_\infty(r,f)=\sup_{|z|=r}|f(z)|.$$

For $-1<\alpha<\infty$ and $0<p<\infty$, the weighted Bergman space $A_\alpha^p(\BB)$( or $A_\alpha^p$) consists of all $f\in H(\BB)$  such that
$$\|f\|_{A_\alpha^p}=  \int_B |f(z)|^pdV_\alpha(z)=c_\alpha\int_B |f(z)|^p(1-|z|^2)^\alpha dV(z)<\infty,$$
 where
$c_\a=\Gamma(n+\a+1)/(\Gamma(n+1)\Gamma(\a+1))$.  When $\alpha=0$, $A^p_0(\BB)=A^p(\BB)$ is the standard Bergman space. It is known that
 $f \in A^p_\a$ if and only if $\Re f(z) \in A_{\alpha+p}^p$. Moreover
  \begr
  \|f\|^p_{ A^p_\alpha}\asymp |f(0)|^p+ \int_\BB |\Re f(z)  |^p(1-|z|^2)^pdV_\alpha(z).   \nonumber
\endr
Here $\Re f $ is the radial derivative of $f,$ i.e., $\Re f(z)=\sum^n_{j=1}z_j \frac{\partial f}{\partial
z_j}(z).$    See \cite{Rw1980,Zk2005} for the theory of $H^p$ and $A_\alpha^p$ in the unit ball.

Suppose $\om$ is a radial weight ($i.e.$, a positive and  integrable  function in $\BB$ such that $\om(z)=\om(|z|)$).  Let $\hat{\om}(r)=\int_r^1\om(t)dt$. $\om$ is called a doubling weight, denoted by $\om\in \hD$, if there is a constant $C>0$ such that
$$\hat{\om}(r)<C\hat{\om}(\frac{1+r}{2}) ,\,\,\mbox{ when } 0\leq r<1.$$
  $\om$ is called a regular weight,  denote by  $\om\in    \R $, if there is a constant $C>0$ determined by   $\om$, such that
$$\frac{1}{C}<\frac{\hat{\om}(r)}{(1-r)\om(r)}<C,\,\,\mbox{ when } 0 \leq r<1.$$
  $\om$ is called a rapidly increasing weight, denote by  $\om\in\I$, if
$$\lim_{r\to 1} \frac{\hat{\om}(r)}{(1-r)\om(r)}=\infty.$$
After a calculation, we see that $\I\cup\R \subset\hD.$   See  \cite{Pja2015,PjaRj2014book} for more details about $\I,\R$,  $\hD$.

In \cite{PjaRj2014book}, J. Pel\'aez and J. R\"atty\"a introduced a new class function space $A_\om^p(\D)$, the weighted Bergman space induced by rapidly increasing weights in $\D$. In \cite{PjaRj2014book}, they investigated some basic properties of $\om$ with $\om\in\R\cup\I$,  described the $q-$Carleson measure for $A_\om^p(\D)$, gave equivalent characterizations of $A_\om^p(\D)$,  characterized   the boundedness, compactness and Schatten classes of  Volterra integral operator $J_g$ on
$A_\om^p(\D)$.  In \cite{Pja2015}, J.  Pel\'aez extended many results from $\om\in\R\cup\I$ to $\om\in\hD$.
See \cite{Pja2015,PjaRj2014book,PjaRj2015,PjaRj2016,PjRj2016jmpa,PjaRjSk2015mz,PjaRjSk2018jga} for many results on $A_\om^p(\D)$ with $\om\in\hD$.

Motivated by \cite{PjaRj2014book}, we extend the Bergman space $A_\om^p(\D)$  with $\om\in\hD$ to the unit ball.
 Let $\om\in\hD$ and $0<p<\infty$. The weighted Bergman space $A_\om^p=A_\om^p(\BB)$ is the space of all $f\in H(\BB)$ for which
$$\|f\|_{A_\om^p}^p=\int_{\BB}|f(z)|^p\om(z)dV(z)<\infty.$$
 It is easy to check that $A_\om^p$ is a Banach space when $p\geq 1$ and a complete metric space with the distance
$\rho(f,g)=\|f-g\|_{A_\om^p}^p$ when $0<p<1$.
When $\om(z)=c_\alpha(1-|z|^2)^\alpha(\alpha>-1)$, the space $A_\om^p$ becomes the classical weighted Bergman space $A_\alpha^p$.

Suppose that $g\in H(\D)$. The integral operator $J_g$,
called the Volterra integral operator, is defined  by
  \begr J_gf(z)=\int_0^z f(\xi)g'(\xi)d\xi , ~ \qquad~~ ~f\in H(\D), ~~  z\in \D. \nonumber
   \endr
 The operator $J_g$ was first introduced  by Pommerenke in \cite{Pc1977cmh}.  He showed that $J_g$ is a bounded operator on the
Hardy space $H^2(\D)$ if and only if $g\in BMOA(\D)$.  See \cite{AaCj2001jam, AaPja2012iumj, AaSa1995cv,AaSa1997iumj} for the study of the boundedness,  compactness and the spectrum of  $J_g$ in $H^p(\D)$ and $A_\alpha^p(\D)$.

Let $g\in H(\BB)$. Define
 $$T_gf(z)=\int_0^1 f(tz)\Re g(tz)\frac{dt}{t}, ~~\qquad~~~~f\in H(\BB), ~~z\in \BB.  $$
This operator is also called the Volterra type integral operator (or the Riemann-Stieltjes operator, or the Extended Ces\`aro
operator).   The operator $T_g$  was  introduced by Z. Hu in \cite{hu1} and studied, for example in \cite{hu1, hu3, lsbbmss, Pj2016jfa}.
In particular, J. Pau completely described the boundedness and compactness of $T_g$ between different Hardy spaces in the unit ball of $\CC^n$ in \cite{Pj2016jfa}.

In this paper, we will investigate some properties of $A_\om^p$  in the unit ball of $\CC^n$  and study the boundedness and compactness of  $T_g:A_\om^p\to A_\om^q$ with $\om\in\hD$ and $0<p\leq q<\infty$,
The paper is organized as follows. In section 2, we recall some well-known results and notations, define (standard) Carleson block $S_a$ for $a\in\BB$ and  estimate the volume of $S_a$.
In section 3, we characterize the $q-$Carleson measure for $A_\om^p$ with standard Carleson block $S_a$ for $a\in\BB$.
In section 4, we  extend the admissible approach region $\Gamma_u$ from $u\in\SS$ to $u\in\overline{\BB}$ by dilation transformation, get some equivalent characterizations for  $A_\om^p$ by using the radial derivative and admissible approach regions.
In section 5, we define a new class of holomorphic functions $\mathcal{C}^\kappa(\om^*)(\kappa\geq 1)$ and the little-oh subspace of it,
and then we investigate the boundedness and compactness of $T_g:A_\om^p\to A_\om^q$ with $0<p\leq q<\infty$ and $\om\in\hD$.
In section 6,  we discuss   the inclusion relationship between $\mathcal{C}^1(\om^*)(\mathcal{C}_0^1(\om^*))$ and some other function spaces,
such as the Bloch space $\B$, the little Bloch space $\B_0$, the $BMOA$ space and the $VMOA$ space.

Throughout this paper, the letter $C$ will denote  constants and may differ from one occurrence to the other.
The notation $A \lesssim B$ means that there is a positive constant C such that $A\leq CB$.
The notation $A \approx B$ means $A\lesssim B$ and $B\lesssim A$.\msk

 \section{ Preliminary results}

 For any $\xi,\tau\in\overline{\BB}$, let $d(\xi,\tau)=|1-\langle \xi,\tau\rangle|^\frac{1}{2}$.
 Then $d(\cdot,\cdot)$ is the nonisotropic metric.
For $r>0$   and $\xi\in\SS$, let
  $$Q(\xi,r)=\{\eta\in \SS:|1-<\xi,\eta>|\leq r^2\}.$$
 $Q(\xi,r)$  is a ball in $\SS$ for all $\xi\in \SS^n$ and $r\in(0,\sqrt{2})$.
More information about $d(\cdot,\cdot)$ and $Q(\xi,r)$ can be found in \cite{Rw1980,Zk2005}.
Lemma 4.6 in \cite{Zk2005} is very useful in this paper, and we express it as follows.

\begin{Lemma}\label{1212-1}
There exist positive constants $A_1$ and $A_2$ (depending on $n$ only) such that
$$A_1\leq \frac{\sigma(Q(\xi,r))}{r^{2n}}\leq A_2$$
for all $\xi\in\SS^n$ and $r\in(0,\sqrt{2})$.
\end{Lemma}

For any $a\in\BB\backslash\{0\}$, let
$Q_a=Q({a}/{|a|},\sqrt{1-|a|}),$
and define
$$S_a=S(Q_{a})=\left\{z\in\BB:\frac{z}{|z|}\in Q_{a},|a|<|z|<1\right\}.  $$
For convince, if $a=0$, let $Q_a=\SS$ and $S_a=\BB$.  We call $S_a$ the Carleson block.
Now we give a estimate for the volume of $S_a$.
As usual, for a measurable set $E\subset\BB$, $\om(E)=\int_E \om(z)dV(z)$.\msk

\begin{Lemma}\label{1210-3}
Assume that $\om\in\hD$, $r\in[0,1]$ and $\om^*(r)=\int_r^1 s\om(s)\log\frac{s}{r}ds$. Then  the following statements hold.
\begin{enumerate}[(i)]
  \item $\om^*\in\R$ and $\om^*(r)\approx (1-r)\int_r^1 \om(t)dt$ when $r\in(\frac{1}{2},1)$;
  \item There are  $0<a<b<+\infty$ and $\delta\in [0,1)$, such that
  $$\frac{\om^*(r)}{(1-r)^a} \;\; \mbox{is decreasing on}\;\; [\delta,1)\;
\mbox{and}\;\; \lim_{r\to 1}\frac{\om^*(r)}{(1-r)^a}=0;$$
$$\frac{\om^*(r)}{(1-r)^b} \;\; \mbox{is increasing on}\;\; [\delta,1)\;\;
\mbox{and}\;\; \lim_{r\to 1}\frac{\om^*(r)}{(1-r)^b}=\infty; $$
\item $\om^*(r)$ is decreasing on $[\delta,1)$ and $\lim\limits_{r\to 1} \om^*(r)=0.$
\item  $\om(S_a)\approx (1-|a|)^n\int_{|a|}^1 \om(r)dr$.
\end{enumerate}
\end{Lemma}

\begin{proof}$(i)-(iii)$ can be found in \cite{PjaRj2014book,PjaRj2015}. From \cite[Lemma 1.8]{Zk2005}, we see that
\begin{align}\label{0122-1}
\int_{\BB}f(z)dV(z)=2n\int_0^1r^{2n-1}dr\int_{\SS}f(r\xi)d\sigma(\xi).
\end{align}
Then by Lemma \ref{1212-1}, we have
\begin{align*}
\om(S_a)=2n\int_{|a|}^1 r^{2n-1}\om(r)dr\int_{Q(\frac{a}{|a|},\sqrt{1-|a|})}d\sigma(\xi)
\approx (1-|a|)^n\hat{\om}(|a|).
\end{align*}
The proof is complete.
\end{proof}

\begin{Lemma}\label{1206-1}
There exists $q=q(n)>1$, such that for all $r\in \left(0,\frac{\sqrt{2}}{q}\right)$ and $\xi\in \SS^n$,
\begin{align*}
\sigma\left({Q(\xi,qr)\backslash Q(\xi,r)}\right)\approx r^{2n}.
\end{align*}
\end{Lemma}

\begin{proof}
 By Lemma \ref{1212-1}, there exist $A_2>A_1>0$, such that
$$A_1 r^{2n}\leq \sigma(Q(\xi,r))\leq A_2  r^{2n},\,\,\mbox{ for all }\,\, r\in (0,\sqrt{2})\,\,\mbox{ and }\,\, \xi\in\SS^n.$$
Fix a $q>1$ such that $A_1 q^{2n}>A_2$. Then we have
$$
\sigma\left({Q(\xi,qr)\backslash Q(\xi,r)}\right)
\geq (A_1 q^{2n}-A_2)r^{2n}\gtrsim r^{2n},
$$
and
$$
\sigma\left({Q(\xi,qr)\backslash Q(\xi,r)}\right)
\leq  (A_2 q^{2n}-A_1)r^{2n}\lesssim r^{2n}.
$$
The proof is complete.
\end{proof}

For $q>0$, if  $\mu$ is a positive Borel measure on $\BB$, $L_\mu^q$ consists of the Lebesgue measurable   functions on $\BB$ such that
$$\|f\|_{L_\mu^q}:=\left(\int_\BB |f(z)|^qd\mu(z)\right)^\frac{1}{q}<\infty.$$
To study the compactness of a linear operator $T$ from $A_\om^p$ to $L_\mu^q$, we need the following  lemma which can be obtained in a standard way.
\begin{Lemma}\label{1210-2}
Suppose $0<p, q<\infty, \om\in \hD$ and $\mu$ is a positive Borel measure on $\BB$. If
$T:A_\om^p\to L_\mu^q$ is linear and bounded, then $T$ is compact if and only if
whenever $\{f_k\}$  is bounded in $A_\om^p$ and $f_k\to 0$ uniformly on compact subsets of $\BB$,
$\lim\limits_{k\to\infty}\|Tf_k\|_{L_\mu}^q =0.$
\end{Lemma}

\begin{Lemma}\label{1201-1}
Suppose $\om\in\hD, 0<\alpha<\infty$. Then there exists a constant $C=C(\alpha,\om,n)$ such that
$$|f(z)|^\alpha\leq C M_\om(|f|^\alpha)(z),$$
for all $f\in H(\BB)$.  Here and henceforth,
$$M_\om(\vp)(z)=\sup_{z\in S_a}\frac{1}{\om(S_a)}\int_{S_a}|\vp(\xi)|\om(\xi)dV(\xi).$$
\end{Lemma}

\begin{proof} Fix $q=q(n)$ such that Lemma \ref{1206-1} holds. Let $r_0=\max\left(\frac{1}{2},1-\frac{1}{q}\right)$.

For any $z\in\BB$ such that   $r_0<|z|<1$, if $\frac{1+|z|}{2}< \rho<1$, let $N$ be the largest nature number such that $q^N\left(1-|z|\right)<1$ and
$$Q_k^\prime:=\left\{\xi\in\SS:\left|1-\langle\xi,\frac{z}{|z|}\rangle\right|<q^k\left(1-\frac{|z|}{\rho}\right)\right\}, $$
for $k=0,1,2,\cdots,N$.
For convenience, let   $Q_{-1}^\p=\O$.
Then, $$Q_0^\prime\subset Q_1^\prime\subset \cdots\subset  Q_N^\prime\subset Q_{N+1}^\prime:=\SS.$$
When $\xi\in Q_{k+1}^\prime\backslash Q_k^\prime$ with $k=1,2,\cdots,N$, we have
\begin{align*}
\left|1-\langle\frac{1}{\rho}z,\xi\rangle\right|
&=\left|1-\langle\xi,\frac{z}{|z|}\rangle+\langle\xi,\frac{z}{|z|}\rangle-\langle \xi,\frac{1}{\rho}z\rangle\right|
\geq (q^k-1)\left(1-\frac{|z|}{\rho}\right).
\end{align*}
Then for all $\xi\in Q_{k+1}^\prime\backslash Q_{k}^\prime$ wiht $k=-1,0,1,\cdots, N$, we have
\begin{align*}
\left|1-\langle\frac{1}{\rho}z,\xi\rangle\right|
\gtrsim q^k \left(1-\frac{|z|}{\rho}\right).
\end{align*}

Since $\om\in\hD$, by Lemma 2.1(ii) in \cite{Pja2015}, there exist $C_0=C_0(\om)\geq 1$ and $\beta=\beta(\om)>0$, such that
\begin{align*}
\hat{\om}(r)\leq C_0\left(\frac{1-r}{1-t}\right)^\beta\hat{\om}(t),\,\mbox{ for all }\,0\leq r\leq t<1.
\end{align*}
Write $\alpha=s\gamma$, where $\gamma>1+\frac{\beta}{n}>1$.
Suppose $\frac{1}{\gamma}+\frac{1}{\gamma^\p}=1$.
Then
by Corollary 4.5 in \cite{Zk2005},  H\"older's inequality and Theorem 1.12 in \cite{Zk2005}, we have
\begin{align*}
|f(z)|^s
&\leq \int_{\SS} \frac{\left(1-\frac{1}{\rho^2}|z|^2\right)^n}{\left|1-\langle\frac{1}{\rho}z,\xi\rangle\right|^{2n}} |f(\rho\xi)|^s d\sigma(\xi)  \\
&\leq \left(\int_{\SS} \frac{\left(1-\frac{1}{\rho^2}|z|^2\right)^{n\gamma-n}}{\left|1-\langle\frac{1}{\rho}z,\xi\rangle\right|^{n\gamma}} |f(\rho\xi)|^{s\gamma} d\sigma(\xi) \right)^\frac{1}{\gamma}
\left(
\int_\SS \frac{\left(1-\frac{1}{\rho^2}|z|^2\right)^{n\gamma^\p-n}}{\left|1-\langle\frac{1}{\rho}z,\xi\rangle\right|^{n\gamma^\p}} d\sigma(\xi)
\right)^\frac{1}{\gamma^\p}
 \\
&\lesssim
\left(\int_{\SS} \frac{\left(1-\frac{1}{\rho^2}|z|^2\right)^{n\gamma-n}}{\left|1-\langle\frac{1}{\rho}z,\xi\rangle\right|^{n\gamma}} |f(\rho\xi)|^{s\gamma} d\sigma(\xi) \right)^\frac{1}{\gamma}.
\end{align*}
Therefore,
\begin{align*}
|f(z)|^\alpha
&\lesssim \int_{\SS} \frac{\left(1-\frac{1}{\rho^2}|z|^2\right)^{n\gamma-n}}{\left|1-\langle\frac{1}{\rho}z,\xi\rangle\right|^{n\gamma}} |f(\rho\xi)|^\alpha d\sigma(\xi)\\
&=\sum_{k=-1}^N\int_{Q_{k+1}^\prime\backslash Q_k^\prime} \frac{\left(1-\frac{1}{\rho^2}|z|^2\right)^{n\gamma-n}}{\left|1-\langle\frac{1}{\rho}z,\xi\rangle\right|^{n\gamma}}
|f(\rho\xi)|^\alpha d\sigma(\xi)   \\
&\lesssim \sum_{k=-1}^N\frac{1}{q^{n\gamma k}\left(1-\frac{|z|}{\rho}\right)^n}\int_{Q_{k+1}^\prime\backslash Q_k^\prime}|f(\rho\xi)|^\alpha d\sigma(\xi)\\
&\lesssim \frac{1}{(1-|z|)^n}\sum_{k=-1}^N\frac{1}{q^{n\gamma k}}\int_{Q_{k+1}^\prime}|f(\rho\xi)|^\alpha d\sigma(\xi).
\end{align*}
When $k=0,1,2,\cdots,N$, let $t_k=1-q^k(1-|z|)$ and $a_k=t_k z$.
When $k=N+1$, let $a_k=0$, $Q_{a_k}=\SS$ and $S_{a_k}=\BB$.
For $0\leq k\leq N+1$,   we have
 $$Q_k^\prime\subset Q_{a_k}\subset \SS,\,\,\,\,  1< \frac{1-|a_k|}{1-|z|}\lesssim q^k,\,\,\,\,
 \frac{1}{\hat{\om}(z)} \leq C_0\left(\frac{1-|a_k|}{1-|z|}\right)^\beta \frac{1}{\hat{\om}(a_k)}.$$
 For all $|z|\geq r_0$, $\int_{|z|}^1  r^{2n-1}\om(r)dr\approx \int_{\frac{|z|+1}{2}}^1 r^{2n-1}\om(r)dr$. So,  we have
\begin{align*}
|f(z)|^\alpha (1-|z|)^{n}\int_{|z|}^1  r^{2n-1}\om(r)dr
\approx &\int_{\frac{|z|+1}{2}}^1  \rho^{2n-1}\om(\rho)|f(z)|^\alpha (1-|z|)^{n}d\rho   \\
\lesssim & \int_{\frac{|z|+1}{2}}^1  \rho^{2n-1}\om(\rho)\left(
\sum_{k=0}^{N+1}\frac{1}{q^{n\gamma k}}\int_{Q_{k}^\prime}|f(\rho\xi)|^\alpha d\sigma(\xi)
\right)d\rho  \\
\lesssim &\sum_{k=0}^{N+1} \frac{1}{q^{n\gamma k}}\int_{S_{a_k}} |f(\zeta)|^\alpha \om(\zeta)dV(\zeta).
\end{align*}
By Lemma \ref{1210-3}, we have
\begin{align*}
\frac{1}{(1-|z|)^{n}\int_{|z|}^1  r^{2n-1}\om(r)dr}
\approx  \frac{1}{(1-|z|)^n\hat{\om}(z)}
\lesssim \frac{(1-|a_k|)^{n+\beta}}{(1-|z|)^{n+\beta} \om(S_{a_k})}
\lesssim \frac{q^{(n+\beta)k}}{\om(S_{a_k})}.
\end{align*}
Then,
\begin{align*}
|f(z)|^\alpha
&\lesssim \sum_{k=0}^{N+1} \frac{1}{q^{(n\gamma-n-\beta) k}}  \frac{\int_{S_{a_k}} |f(\xi)|^\alpha \om(\xi)dV(\xi)}{\om(S_{a_k})}
\lesssim M_\om(|f|^\alpha)(z).
\end{align*}

Next we  suppose  that $|z|\leq r_0$. For all $a\in\BB$ such that $z\in S_a$, we have $|a|<|z|\leq r_0$.
By Lemma \ref{1210-3}, $\om(S_a)\approx 1$. Then,
\begin{align*}
\sup_{z\in S_a}\frac{1}{\om(S_a)}\int_{S_a}|f(\xi)|^\alpha\om(\xi)dV(\xi)
&\approx \sup_{z\in S_a}\int_{S_a}|f(\xi)|^\alpha\om(\xi)dV(\xi)\\
& =\int_{\BB}|f(\xi)|^\alpha\om(\xi)dV(\xi).
\end{align*}
Using Cauchy's formula, we have
$$|f(z)|^\alpha
\lesssim \int_{\BB}|f(\xi)|^\alpha\om(\xi)dV(\xi)
\approx \sup_{z\in S_a}\frac{1}{\om(S_a)}\int_{S_a}|f(\xi)|^\alpha\om(\xi)dV(\xi).$$
The proof is complete.
\end{proof}

Here and henceforth, for all $a\in\BB$ and $0<p<\infty$, set
\begin{align}\label{0119-2}
F_{a,p}=\left(\frac{1-|a|^2}{1-\langle z,a\rangle}\right)^{\frac{\gamma+n}{p}}.
\end{align}

We obtain the following lemma.

\begin{Lemma}\label{1208-4} Suppose $\om\in\hD, 0<p<\infty$ and $\gamma$ is large enough.  For all $a\in\BB$,
\begin{align}\label{1208-5}
|F_{a,p}(z)|\approx 1, \,\,z\in S_a,
\end{align}
and
\begin{align*}
\|F_{a,p}\|_{A_\om^p}^p \approx \om(S_a).
\end{align*}
\end{Lemma}

\begin{proof}For all $z\in S_a$, we have
$$
\frac{1-|a|}{|1-\langle z,a\rangle|}\leq \frac{1-|a|}{1-|\langle z,a\rangle|}\leq \frac{1-|a|}{1-|a| |z|}\leq 1
$$
and
\begin{align*}
\frac{1-|a|}{|1-\langle z,a\rangle|}
&\geq \frac{1-|a|}{\left|1-\langle\frac{z}{|z|},\frac{a}{|a|} \rangle\right|  + |\langle z,a\rangle| \left|\frac{1}{|z| |a|}-1\right|}
\geq \frac{1-|a|}{1-|a|+1-|a|^2}\gtrsim 1.
\end{align*}
It follows that (\ref{1208-5}) holds.

By Lemmas \ref{1212-1} and \ref{1210-3}, we have
\begin{align}\label{1208-8}
\sigma(Q_a)\approx (1-|a|)^{n},
\,\,\mbox{ and }\,\,
 \om(S_a)\approx (1-|a|)^{n}  \int_{|a|}^1  \om(t)dt.
 \end{align}
 By Theorem 1.12 in \cite{Zk2005}  and Lemma 2.1(iii) in \cite{Pja2015},  if $\gamma $ is large enough, we have
 \begin{align*}
 \|F_{a,p}\|_{A_\om^p}^p
& = 2n (1-|a|)^{\gamma+n}\int_0^1\om(r)  r^{2n-1}\int_{\SS} \frac{1}{|1-\langle r\xi,a\rangle|^{\gamma+n}}d\sigma(\xi)     dr \\
& \approx 2n (1-|a|)^{\gamma+n}\int_0^1  \frac{r^{2n-1}\om(r)}{(1-r|a|)^{\gamma}} dr \\
& \leq 2n (1-|a|)^{\gamma+n}\left(\int_0^{|a|}  \frac{\om(r)}{(1-r)^{\gamma}} dr
 + \int_{|a|}^{1}  \frac{\om(r)}{(1-|a|)^{\gamma}} dr \right)\\
& \lesssim  (1-|a|)^{n}\int_{|a|}^1  \om(r) dr \approx \om(S_a).
 \end{align*}
By (\ref{1208-5}), $\|F_{a,p}\|_{A_\om^p}^p\gtrsim \om(S_a)$ is obvious. The proof is complete.
\end{proof}

In the rest  this paper, we always assume $F_{a,p}$ satisfies the condition of Lemma \ref{1208-4}.

In the last of this section, we define a $\alpha-$Carleson block $S_{a,\alpha}$ for all  $a\in\BB\backslash\{0\}$ and any fixed $\alpha>-1$. That is
\begin{align*}
S_{a,\alpha}=\left\{z\in\BB:|a|<|z|<1, \left|1-\langle\frac{z}{|z|},\frac{a}{|a|} \rangle\right|\leq (\alpha+1)(1-|a|)\right\}.
\end{align*}
When $a=0$, we define $S_{a,\alpha}=\BB$. Obviously, for all $a\in\BB$, we have $S_{a,0}=S_a$ and $S_a\subset S_{a,\alpha}(\alpha> 0 )$.
The following proposition is useful in this paper.

\begin{Proposition}\label{0118-1} For any fixed $\alpha\geq 0$, there exist $N\in\N$, such that, for all $a\in\BB$, there are $a_1,a_2,\cdots, a_k$ satisfy the following condition:
\begin{enumerate}[(i)]
  \item $k\leq N$ and $|a_1|=|a_2|=\cdots=|a_k|=|a|$;
  \item $S_{a,\alpha}\subset \cup_{i=1}^k S_{a_i}$.
\end{enumerate}
\end{Proposition}

\begin{proof} Suppose $a\in \BB\backslash\{0\}$ is fixed. For any $\tau\in \SS$, define
$$E_\tau=Q(\tau,\frac{1}{2}\sqrt{1-|a|}),\,\,
 \mbox{ and }\,\,
 E_a^\p=Q(\frac{a}{|a|},(\frac{1}{2}+\sqrt{1+\alpha})\sqrt{1-|a|}).$$

Since $\frac{\sigma(E_a^\p)}{\sigma(E_\tau)}<\infty$, there are at most $M:=M(a)$ elements $\xi_1,\xi_2,\cdots, \xi_{M}$ in $\SS$ such that
 \begin{enumerate}[(a)]
   \item $E_{\xi_i}\cap E_{\xi_j}=\O$ for all $1\leq i <j\leq M$;
   \item $E_{\xi_i}\subset E_a^\p$ for all $1\leq i \leq M$.
 \end{enumerate}
 Then we  have
 $$Q(\frac{a}{|a|},\sqrt{(1+\alpha)(1-|a|)})\subset\cup_{i=1}^M Q(\xi_i,\sqrt{1-|a|}).$$
 Otherwise, there is a $\xi\in Q(\frac{a}{|a|},\sqrt{(1+\alpha)(1-|a|)})$ but $\xi\not\in \cup_{i=1}^M Q(\xi_i,\sqrt{1-|a|})$.
 Then for any $\eta\in E_\xi$, we have
 $$d(\eta,\xi_i)\geq d(\xi_i,\xi)- d(\eta,\xi)> \sqrt{1-|a|}-\frac{1}{2}\sqrt{1-|a|}=\frac{1}{2}\sqrt{1-|a|},$$
 and
 $$d(\eta,\frac{a}{|a|})\leq d(\eta,\xi)+d(\xi,\frac{a}{|a|})<\frac{1}{2}\sqrt{1-|a|}+\sqrt{(1+\alpha)(1-|a|)}.  $$
 That is a contraction with $M$ is the maximum number. By Lemma 1, we have
 $$M\leq \frac{\sigma(E_a^\p)}{\sigma(E_\tau)}\lesssim 1.$$
 Then by letting $a_i=|a|\xi_i$, we finish the proof.
\end{proof}

\begin{Remark}\label{0308-1}
By Lemma \ref{1210-3}, for any fixed $\alpha>0$, $\om(S_a)\approx\om(S_{a,\alpha})$. Hence, many results described by Carleson block also hold for $\alpha$-Carleson block.
\end{Remark}

For $\xi\in\SS$ and $r>0$, a Carleson tube $S^*(\xi,r)$ can be define as
$$S^*(\xi,r)=\{z\in\BB:|1-\langle z,\xi\rangle|<r\}.$$
As we know, Carleson tube is very useful in the study of the function space on the unit ball of $\mathbb{C}^n$.
For  the convenience, we often restrict $0<r<\delta$ for some $\delta>0$. Here, we will compare Carleson tube with Carleson block.

\begin{Proposition}\label{0118-2}
The following assertions hold.
\begin{enumerate}[(i)]
  \item For any $\xi\in \SS$ and $0<r<1$, there exists $a\in\BB$ such that $S^*(\xi,r)\subset S_{a,2}$.
  \item For any $a\in \BB$ with $|a|>\frac{1}{2}$, there exist $\xi\in \SS$ and $0<r<1$ such that $ S_{a}\subset S^*(\xi,r)$.
\end{enumerate}
\end{Proposition}

\begin{proof}
(i).  For any $z\in S^*(\xi,r)$ with $0<r<1$, by letting $a=(1-r)\xi$, we have $|z|>|a|$ and
\begin{align*}
\left|1-\langle \frac{z}{|z|},\xi \rangle\right|
&\leq \left|1-\langle z,\xi \rangle\right| +\left|\langle z,\xi \rangle-\langle \frac{z}{|z|},\xi \rangle\right| \leq r+1-|z|<2(1-|a|).
\end{align*}
Then $S^*(\xi,r)\subset S_{a,2}$.

(ii).  Suppose $a\neq 0$. Let $\xi=\frac{a}{|a|}$ and $2(1-|a|)<r<1$.   For any $z\in S_a$, we have
\begin{align*}
\left|1-\langle z,\xi\rangle\right|
&\leq \left|1-\langle \frac{z}{|z|},\frac{a}{|a|}\rangle\right| + \left|\langle \frac{z}{|z|},\frac{a}{|a|}\rangle-\langle z,\xi\rangle\rangle\right|    \\
&\leq 1-|a| +1-|z|\leq 2(1-|a|)<r.
\end{align*}
Then  $ S_{a}\subset S^*(\xi,r)$.
The proof is complete.
\end{proof}\msk

 \section{The $q$-Carleson Measure for $A_\om^p$}
In this section, we  give some descriptions of $q$-Carleson measure for $A_\om^p$ when $0<p\leq q<\infty$. For a given Banach space (or a complete metric space) $X$ of analytic functions on $\BB$, a positive Borel measure $\mu$ on $\BB$ is called a $q$-Carleson measure for $X$ if the identity operator $Id : X \to L^q_\mu$ is bounded.
Moveover, if  $Id : X \to L^q_\mu$ is compact, then we say that $\mu$ is a vanishing $q$-Carleson measure for $X$. \msk

\begin{Theorem}\label{0109-7}
Let $0<p\leq q<\infty$, $\om\in \hD$, and $\mu$ be a positive Borel measure on $\BB$. The the following statements hold:
\begin{enumerate}[(i)]
  \item  $\mu$ is a $q$-Carleson measure for $A_\om^p$ if and only if
  \begin{align}\label{0109-8}
  \sup_{a\in\BB} \frac{\mu(S_a)}{(\om(S_a))^{\frac{q}{p}}}<\infty.
  \end{align}
  Moreover, if $\mu$  is a $q$-Carleson measure for $A_\om^p$, then the identity operator $Id:A_\om^p\to L_\mu^q$ satisfies
  $$\|Id\|_{A_\om^p\to L_\mu^q}^q \approx \sup_{a\in\BB} \frac{\mu(S_a)}{(\om(S_a))^{\frac{q}{p}}}.$$
  \item  $\mu$ is a vanishing $q$-Carleson measure for $A_\om^p$  if and only if
  $$\lim_{|a|\to 1}   \frac{\mu(S_a)}{(\om(S_a))^{\frac{q}{p}}}=0.$$
\end{enumerate}
\end{Theorem}

\begin{proof} First assume that $\mu$ is a $q$-Carleson measure for $A_\om^p$,  By Lemma \ref{1208-4}, we have
\begin{align*}
\mu(S_a)
&\approx \int_{S_a} |F_{a,p}|^q d\mu(z)
\leq \|F_{a,p}\|_{L_\mu^q}^q   \\
&\leq \|Id\|_{A_\om^p\to L_\mu^q}^q \|F_{a,p}\|_{A_\om^p}^q
\approx \|Id\|_{A_\om^p\to L_\mu^q}^q (\om(S_a))^{\frac{q}{p}}.
\end{align*}
So,
$$\sup_{a\in\BB} \frac{\mu(S_a)}{ (\om(S_a))^{\frac{q}{p}}}\lesssim \|Id\|_{A_\om^p\to L_\mu^q}^q.$$

Conversely, suppose $M:=\sup\limits_{a\in\BB} \frac{\mu(S_a)}{(\om(S_a))^{\frac{q}{p}}}<\infty.$ We begin with proving that there exists a constant $K=K(p,q,\om)$ sucht that
\begin{align}\label{1208-7}
\mu(E_s)\leq KMs^{-\frac{q}{p}} \|\vp\|_{L_\om^1}^{\frac{q}{p}},
\end{align}
is valid for all $\vp\in L_\om^1$ and $0<s<\infty.$ Here $ E_s=\{z\in \BB: M_\om(\vp)(z)>s\}$.

If $E_s=\O$, (\ref{1208-7}) holds. If $E_s\neq \O$, define $A_s^\varepsilon$ and $B_s^\varepsilon$ for each $\varepsilon>0$ as follows.
$$A_s^\varepsilon=\left\{z\in\BB: \int_{S_z} |\vp(\xi)|\om(\xi)dV(\xi)>s(\varepsilon+\om(S_z))\right\},$$
and
$$B_s^\epsilon=\left\{z\in\BB: S_z\subset S_u \mbox{ for some }  u\in A_s^\varepsilon\right\}.$$
Then we have
\begin{align*}
E_s= \cup_{\varepsilon>0} B_s^\varepsilon, \mbox{ and } \mu(E_s)=\lim_{\varepsilon\to 0^+} \mu(B_s^\varepsilon).
\end{align*}
Let $E\subset A_s^\varepsilon$ such that for all $\xi,\eta\in E$ we have either $\xi  =\eta $ or $Q_{\xi}\cap Q_{\eta}=\O$. Since
\begin{align}\label{0109-2}
s\sum_{\xi\in E}(\varepsilon +\om(S_\xi))\leq \sup_{\xi\in E}\int_{S_\xi}|\vp(z)|\om(z)dV(z)\leq \|\vp\|_{L_\om^1},
\end{align}
we get that there are only finite elements in $E$. By Lemma 5.6 in \cite{Zk2005}, there are   $\{z_1,z_2,\cdots,z_m\}\subset A_s^\varepsilon$ such that
$Q_{z_j}(1\leq j\leq m)$ are disjoint and
\begin{align}\label{0109-1}
A_s^\varepsilon\subset \cup_{k=1}^m\left\{z\in\BB: Q_z\subset Q_{z_k}^\prime\right\},
\end{align}
 where
$$Q_{z_k}^\prime=Q\left(\frac{z_k}{|z_k|},5\sqrt{1-|z_k|}\right).$$
For any $z\in B_s^\varepsilon$, there is a $u\in A_s^\varepsilon$ such that $S_z\subset S_u$. So, $Q_z\subset Q_u$.
By (\ref{0109-1}), we have
\begin{align}\label{0109-3}
B_s^\varepsilon\subset \cup_{k=1}^m\left\{z\in\BB: Q_z\subset Q_{z_k}^\prime\right\}.
\end{align}
Let $r_k=1-25(1-|z_k|)$. If $r_k>0$, let $z_k^\p=\frac{r_kz_k}{|z_k|}$, and otherwise, let $z_k^\p=0$. Then we have $Q_{z_k}^\p\subset Q_{z_k^\p}$.
Therefore,
\begin{align*}
\mu\left(\left\{z\in\BB: Q_z\subset Q_{z_k}^\prime\right\}\right)
&\leq \mu \left(\left\{z\in\BB: Q_z\subset Q_{z_k^\p}\right\}\right)\\
&=\mu \left(\left\{z\in\BB: S_z\subset S_{z_k^\p}\right\}\right) \\
&\leq \mu \left( S_{z_k^\prime}\right)
\leq M \left(\om(S_{z_k^\prime})\right)^{\frac{q}{p}}
 \approx M  \left(\om(S_{z_k})\right)^{\frac{q}{p}}.
\end{align*}
Here, the last equivalent relation can be get by Lemma \ref{1210-3} and $\om\in\hD$.
Then, by (\ref{0109-2}) and (\ref{0109-3}), we have
\begin{align*}
\mu(B_s^\varepsilon)
&\leq \sum_{k=1}^m \mu\left(\left\{z\in\BB: Q_z\subset Q_{z_k}^\prime\right\}\right)  \\
&\lesssim M \left(\sum_{k=1}^m \om(S_{z_k})\right)^\frac{q}{p}
\lesssim M s^{-\frac{q}{p}}\|\vp\|_{L_\om^1}^\frac{q}{p}.
\end{align*}
Let $\varepsilon\to 0^+$, we have $K=K(p,q,\om)$ such that
$$\mu(E_s)\leq KM s^{-\frac{q}{p}}\|\vp\|_{L_\om^1}^\frac{q}{p}.$$
Then we obtain (\ref{1208-7}).

Next, we will show that $\mu$ is a $q-$Carleson measure for $A_\om^p$.
The  proof is similar to the proof of \cite[Theorem 2.1]{PjaRj2014book}, but for the benefits of the readers and the completeness of the paper, we give the details of the proof. 

Fix $\alpha>\frac{1}{p}$ and let $f\in A_\om^p$. For $s>0$, let  $$|f|^{\frac{1}{\alpha}}=\psi_{\frac{1}{\alpha},s}+\chi_{\frac{1}{\alpha},s},$$
where
$$
\psi_{\frac{1}{\alpha},s}(z)=\left\{
\begin{array}{cc}
  |f(z)|^\frac{1}{\alpha},  & \mbox{ if } |f(z)|^\frac{1}{\alpha}>\frac{s}{2K} \\
  0, &  \mbox{ otherwise }
\end{array}
\right.
$$
and $K$ is the constant in (\ref{1208-7}) such that $K\geq 1$.
Since $p>\frac{1}{\alpha}$, the function $\psi_{\frac{1}{\alpha},s}$ belongs to $L_\om^1$ for all $s>0$,  and
$$M_\om(|f|^\frac{1}{\alpha})\leq M_\om(\psi_{\frac{1}{\alpha},s})+M_\om(\chi_{\frac{1}{\alpha},s})\leq M_\om(\psi_{\frac{1}{\alpha},s})+\frac{s}{2K}.$$
Then,
\begin{align}\label{1210-1}
\left\{
z\in\BB:M_\om(|f|^\frac{1}{\alpha})(z)>s
\right\}
\subset
\left\{
z\in\BB:M_\om(\psi_{\frac{1}{\alpha},s})(z)>\frac{s}{2}
\right\}.
\end{align}
Using Lemma \ref{1201-1}, (\ref{1210-1}), (\ref{1208-7}) and Minkowski's inequality (Fubini's Theorem in the case $p=q$) in order,
  we have
\begin{align}
\int_{\BB}|f(z)|^q d\mu(z)
&\lesssim \int_{\BB} \left(M_\om(|f|^\frac{1}{\alpha})(z)\right) ^{q\alpha}d\mu(z)
\lesssim M\left(\int_{\BB} |f(z)|^p\om(z)dV(z)\right)^\frac{q}{p}.  \label{0109-5}
\end{align}
So, we get $\|Id\|_{A_\om^p\to A_\om^q}^q \lesssim M$. 
The proof of  $(i)$ is complete.

Next we prove $(ii)$. First we suppose that $\mu$ is a vanishing $q$-Carleson measure for $A_\om^p$. Let
$$f_{a,p}(z)=\left(\frac{1-|a|^2}{1-\langle z,a\rangle}\right)^{\frac{\gamma+n}{p}}
\frac{1}{\left(\om(S_a)\right)^\frac{1}{p}},
 $$
for some $\gamma$ is large enough. By Lemmas \ref{1210-3} and  \ref{1208-4} $f_{a,p}$ is bounded in $A_\om^p$ and converges to 0 uniformly on compact subset of $\BB$ as $|a|\to 1$. By Lemma \ref{1210-2}, we have $\lim\limits_{|a|\to 1}\|f_{a,p}\|_{L_\mu^q}=0.$ Since
\begin{align*}
\|f_{a,p}\|_{L_\mu^q}^q
\geq \int_{S_a} |f(z)|^q d\mu(z)\approx \frac{\mu(S_a)}{(\om(S_a))^\frac{q}{p}},
\end{align*}
we have $\lim\limits_{|a|\to 1} \frac{\mu(S_a)}{(\om(S_a))^\frac{q}{p}}=0.$

Conversely we suppose that $\lim\limits_{|a|\to 1} \frac{\mu(S_a)}{(\om(S_a))^\frac{q}{p}}=0.$ For all $\varepsilon>0$, there exists $r=r(\varepsilon)\in(0,1)$ such that when $|a|>r$, $\frac{\mu(S_a)}{(\om(S_a))^\frac{q}{p}}<\varepsilon$. Let $d\mu_r(z)=\chi_{r\leq |z|<1}d\mu(z)$.

If $|a|\geq r$, $\mu_r(S_a)=\mu(S_a)$. Then suppose $0< |a|<r$. Since $\sigma(\SS)<\infty$ and $$\sigma(Q(\xi,\frac{\sqrt{1-r}}{2}))\approx (1-r)^n>0,$$
by the proof of Proposition \ref{0118-1}, for all $\xi\in \SS$, there are at most $N$ elements $\xi_1,\xi_2,\cdots,\xi_N$ in $Q_a$ such that,
$$Q_a\subset\cup_{i=1}^N Q(\xi_i,\sqrt{1-r}),$$
and
$$N\lesssim  \left(\frac{\sqrt{1-|a|}+\frac{\sqrt{1-r}}{2}}{\frac{\sqrt{1-r}}{2}}\right)^{2n}\approx \left(\frac{1-|a|}{1-r}\right)^n.$$
Therefore,
$$
E_a:=\left\{z\in S_a:r<|z|<1\right\}
\subset \cup_{k=1}^N \left\{z\in\BB:r<|z|<1,\frac{z}{|z|}\in Q(\xi_i,\sqrt{1-r})\right\}.
$$
Since
$$\left\{z\in\BB:r<|z|<1,\frac{z}{|z|}\in Q(\xi_i,\sqrt{1-r})\right\}=S_{r\xi_i},$$
 by Lemma \ref{1210-3}, we have
\begin{align*}
\mu_r(S_a)=\mu(E_a) &\leq \sum_{i=1}^N\mu(S_{r\xi_i})
\leq \varepsilon \sum_{i=1}^N\left(\om(S_{r\xi_i})\right)^{\frac{q}{p}}\\
&\approx N\varepsilon  (1-r)^{\frac{nq}{p}} \left(\int_r^1 \om(t)dt\right)^\frac{q}{p}
\end{align*}
and
\begin{align}\label{0122-2}\frac{\mu_r(S_a)}{(\om(S_a))^\frac{q}{p}}
\lesssim \varepsilon
\left(\frac{1-r}{1-|a|}\right)^{\frac{nq}{p}-n}
\left(\frac{\int_r^1\om(t)dt}{\int_{|a|}^1\om(t)dt}\right)^\frac{q}{p}\leq \varepsilon.
\end{align}
Then, $\|Id\|_{A_\om^p\to L_{\mu_r}^q}^q\lesssim \varepsilon$.

So, if $\{f_k\}$ is bounded in $A_\om^p$ and converges to 0 uniformly on compact subset of $\BB$, then we have
\begin{align*}
\limsup_{k\to\infty}\|f_k\|_{L_\mu^q}^q
&=\limsup_{k\to\infty}\left(\int_{r\BB} |f_k(z)|^q d\mu(z)+\int_{\BB}|f_k(z)|^qd\mu_r(z) \right) \\
&=\limsup_{k\to \infty} \|f_k\|_{L_{\mu_r}^q}^q\lesssim  \varepsilon \limsup_{k\to \infty} \|f_k\|_{A_\om^p}^q.
\end{align*}
Since $\varepsilon$ is arbitrary and $\sup_{k\to \infty} \|f_k\|_{A_\om^p}<\infty$, $\lim\limits_{k\to\infty}\|f_k\|_{L_\mu^q}=0$. So, $\mu$ is a vanishing $q$-Carleson measure for $A_\om^p$. The proof is complete.
\end{proof}

As a by-product of the proof of Theorem \ref{0109-7}, we have the following result which is of independent interest.

\begin{Corollary}\label{0119-3} Let $0 < p \leq q < \infty$ and $0 < \alpha < \infty$ such that $p\alpha > 1$.
Let $\mu$ be a positive Borel measure on $\BB$ and $\om\in\hD$. Then $[M_\om((\cdot)^\frac{1}{\alpha})]^\alpha:
L_\om^p\to L_\mu^q$ is bounded if and only if (\ref{0109-8}) holds. Moreover,
  $$\|[M_\om((\cdot)^\frac{1}{\alpha})]^\alpha\|_{L_\om^p\to L_\mu^q}^q \approx \sup_{a\in\BB} \frac{\mu(S_a)}{(\om(S_a))^{\frac{q}{p}}}.$$
\end{Corollary}
\begin{proof}
By 
(\ref{0109-5}), we obtain
$$\|[M_\om((\cdot)^\frac{1}{\alpha})]^\alpha\|_{L_\om^p\to L_\mu^q}^q \lesssim \sup_{a\in\BB} \frac{\mu(S_a)}{(\om(S_a))^{\frac{q}{p}}}.$$
By Theorem \ref{0109-7} and Lemma \ref{1201-1}, we have
\begin{align*}
\sup_{a\in\BB} \frac{\mu(S_a)}{(\om(S_a))^{\frac{q}{p}}}
&\approx \sup_{f\in A_\om^p} \frac{\|f\|_{L_\mu^q}^q}{\|f\|_{A_\om^p}^q}  \lesssim \sup_{f\in A_\om^p} \frac{\|[M_\om(|f|^\alpha)]^\frac{1}{\alpha}\|_{L_\mu^q}^q}{\|f\|_{L_\om^p}^q}
\\
&\leq \|[M_\om((\cdot)^\frac{1}{\alpha})]^\alpha\|_{L_\om^p\to L_\mu^q}^q .
\end{align*}
The proof is complete.
\end{proof}\msk

\section{Equivalent norms for $A_\om^p$ space}

In this section, we  give some   equivalent norms for the space $A_\om^p$ on the unit ball. These norms
  are inherited from different equivalent $H^p$ norms. First, we give some notations.

Let $\alpha>2$. The admissible approach region $\Gamma_{\zeta,\alpha}$ for some  $\zeta\in\overline{\BB}\backslash\{0\}$ can be defined as
$$\Gamma_{\zeta,\alpha}=\left\{z\in\BB:\left|1-\langle z,\frac{\zeta}{|\zeta|^2}\rangle\right|
<\frac{\alpha}{2}\left(1-\frac{|z|^2}{|\zeta|^2}\right)\right\}.$$
When $\zeta=0$, let $\Gamma_{\zeta,\alpha}=\{0\}$. Obviously, if $r>0$ and $r\zeta,\zeta\in\overline{\BB}$,
$z\in\Gamma_{\zeta,\alpha}$ if and only if $ rz\in\Gamma_{r\zeta,\alpha}.$
Define
$$T_{z,\alpha}=\{\zeta\in\BB:z\in \Gamma_{\zeta,\alpha}\}.$$

It follows from Fubini's Theorem, for a positive function $\vp$ and a finite positive measure $\mu$, one has
\begin{align*}
\int_{\BB}\vp(z)d\mu(z)\approx\int_{\SS}\left(\int_{\Gamma_{\eta,\alpha}}\vp(z)\frac{d\mu(z)}{(1-|z|^2)^n}\right)d\sigma(\eta).
\end{align*}
See \cite{Pj2016jfa}, for example. This fact will be used frequently in this paper.

\begin{Proposition}\label{0110-1}
Suppose $\alpha>2$ is fixed and $\om\in\hD$. Then we have the following statements.
\begin{enumerate}[(i)]
  \item  $  T_{z,\alpha}\subset S_{z,\alpha}$.
  \item There exist  $r=r(\alpha)$ and $\beta>-1$, such that $S_{\frac{1+|z|}{2|z|}z,\beta}\subset T_{z,\alpha}$ when $|z|>r$.
  \item $\om(T_{z,\alpha})\approx \om(S_{z,0}).$
\end{enumerate}
\end{Proposition}
\begin{proof}
{\bf (i). }Suppose $\zeta\in T_z$, that is,
$$\left|1-\langle z,\frac{\zeta}{|\zeta|^2}\rangle\right|
<\frac{\alpha}{2}\left(1-\frac{|z|^2}{|\zeta|^2}\right).$$
So, we have $|\zeta|> |z|$  and
\begin{align*}
\left|1-\langle \frac{\zeta}{|\zeta|},\frac{z}{|z|} \rangle\right|
&\leq \left| 1-\langle z,\frac{\zeta}{|\zeta|^2} \rangle \right|+\left|\langle z,\frac{\zeta}{|\zeta|^2} \rangle-\langle z,\frac{\zeta}{|\zeta| |z|}\rangle\right|
\leq (\alpha +1)\left(1-|z|\right).
\end{align*}
Therefore, $\zeta\in S_{z,\alpha}$, i.e.  $T_z\subset S_{z,\alpha}$.

{\bf (ii). } Suppose $\zeta\in S_{\frac{1+|z|}{2|z|}z,\beta}$. Then we have $|\zeta|>\frac{1+|z|}{2}$ and
\begin{align*}
\left|1-\langle z,\frac{\zeta}{|\zeta|^2}\rangle\right|
&\leq \left| 1-\langle \frac{z}{|z|},\frac{\zeta}{|\zeta|} \rangle \right|+\left|\langle \frac{z}{|z|},\frac{\zeta}{|\zeta|} \rangle-\langle z,\frac{\zeta}{|\zeta|^2}\rangle\right|  \\
&\leq \frac{\beta+1}{2}(1-|z|)+\left(1-\frac{|z|}{|\zeta|}\right)\\
&\leq \left(\frac{\beta+1}{2}\frac{1-|z|}{1-\frac{|z|^2}{|\zeta|^2}}+1\right)\left(1-\frac{|z|^2}{|\zeta|^2}\right).
\end{align*}
Since $1-\frac{|z|^2}{|\zeta|^2}>\frac{(1-|z|)(3|z|+1)}{(1+|z|)^2}$,
\begin{align*}
\left|1-\langle z,\frac{\zeta}{|\zeta|^2}\rangle\right|
&\leq \left(\frac{\beta+1}{2}\frac{(1+|z|)^2}{3|z|+1}+1\right)\left(1-\frac{|z|^2}{|\zeta|^2}\right).
\end{align*}
Let $\beta=\frac{2\alpha-7}{3}$ and  $r(\alpha)\in(0,1)$
such that
$$\frac{(1+|z|)^2}{3|z|+1}<\frac{3}{2}\,\,\mbox{for all}\,\,|z|>r(\alpha).$$
So, if $|z|>r(\alpha)$, we have
\begin{align*}
\left|1-\langle z,\frac{\zeta}{|\zeta|^2}\rangle\right|
&\leq \frac{\alpha}{2}\left(1-\frac{|z|^2}{|\zeta|^2}\right).
\end{align*}
That is, (ii) holds.

{\bf (iii).}
By (i), (\ref{0122-1}) and Lemma \ref{1212-1}, we have
\begin{align}
\om(T_{z,\alpha})\leq
\om(S_{z,\alpha})
&=2n\int_{|z|}^1 r^{2n-1}\om(r)dr\int_{Q(\frac{z}{|z|},\sqrt{(\alpha+1)(1-|z|)})}d\sigma(\xi)  \nonumber\\
&\approx (1-|a|)^n\hat{\om}(|a|)\approx \om(S_{z,0}).\label{0622-1}
\end{align}
Let $\beta$ and $r(\alpha)$ be fixed as in the proof of (ii). When $|z|>r$,
similarly to the proof of (\ref{0622-1}), we have
$$ \om(T_{z,\alpha}) \geq \om(S_{\frac{1+|z|}{2|z|z},\beta})\approx \om(S_{z,0}).$$
When $0<|z|\leq r$, let
$$E_z=\left\{\zeta\in\BB:\frac{1+r}{2}<|\zeta|<1, \left|1-\langle \frac{z}{|z|},\frac{\zeta}{|\zeta|}  \rangle\right|<\frac{\alpha-2}{2}\frac{1-r}{1+r}\right\}.$$
For any $\zeta\in E_z$, we have
\begin{align*}
\left|1-\langle z,\frac{\zeta}{|\zeta|^2}\rangle\right|
&\leq \left|1-\langle \frac{z}{|z|},\frac{\zeta}{|\zeta|}\rangle\right|+ \left|\langle \frac{z}{|z|},\frac{\zeta}{|\zeta|}\rangle-\langle z,\frac{\zeta}{|\zeta|^2}\rangle\right|     \\
&\leq \frac{\alpha-2}{2}\frac{1-r}{1+r}  +  1-\frac{|z|}{|\zeta|}
\leq \frac{\alpha}{2}\left(1-\frac{|z|}{|\zeta|} \right)
<\frac{\alpha}{2}\left(1-\frac{|z|^2}{|\zeta|^2}\right),
\end{align*}
that is, $E_z\subset T_{z,\alpha}$.
Therefore, for all $0<|z|\leq r$, we have
$$\om(T_{z,\alpha})\geq \om(E_z)\approx 1 \approx \om(S_{z,0}).$$
So, (iv) holds. The proof is complete.
\end{proof}

In the rest of this paper, for simplicity, we write $\Gamma_{\eta,\alpha}$ and $T_{z,\alpha}$ by $\Gamma_\eta$ and $T_z$, respectively.
Moreover, if $\om\in\hD$ and $z\in \BB\backslash\{0\}$, let
$$\om^{n*}(z)=\int_{|z|}^1  r^{2n-1}\log\frac{r}{|z|}\om(r)dr.$$
The main result in this section is the following theorem.

\begin{Theorem}\label{0113-2}
Let $0<p<\infty$ and $\om$ be a radial weight. Then
\begin{align}
\|f-f(0)\|_{A_\om^p}^p
&=p^2  \int_{\BB}\frac{|\Re f(z)|^2 |f(z)-f(0)|^{p-2}}{|z|^{2n}}  \om^{n*}(z)dV(z)  \label{0112-1}\\
&\approx \int_{\BB}\left(\int_{\Gamma_u} |\Re f(\xi)|^2\left(1-\frac{|\xi|^2}{|u|^2}\right)^{1-n} dV(\xi)\right)^{\frac{p}{2}} \om(u) dV(u).   \label{0112-3}
\end{align}
Moreover, if $p\geq 2$ and $\om\in\hD$,
\begin{align}
\|f-f(0)\|_{A_\om^p}^p&\approx \int_{\BB}|\Re f(z)|^2|f(z)-f(0)|^{p-2}   \om^{*}(z)dV(z)  \label{0112-4}\\
&\approx \int_{\BB}|\Re f(z)|^2|f(z)-f(0)|^{p-2}   (1-|z|)\hat{\om}(z)dV(z).  \label{0624-1}
\end{align}
\end{Theorem}
\begin{proof}For $0<r<1$, let $f_r(z)=f(rz)$.
By Theorem 4.22 in \cite{Zk2005}, we have
$$
\|f-f(0)\|_{H^p}^p=\frac{p^2}{2n}\int_{\BB} |\Re f(z)|^2 |f(z)-f(0)|^{p-2}|z|^{-2n}\log\frac{1}{|z|}dV(z).
$$
In the following, we always suppose $f(0)=0$.  Then Fubini's Theorem yields
\begin{align}
\|f\|_{A_\om^p}^p
=& 2n\int_0^1 \|f_r\|_{H^p}^p\om(r)r^{2n-1}dr  \nonumber\\
=& p^2\int_0^1 \left(\int_{\BB} |\Re f(rz)|^2 |f(rz)|^{p-2}|z|^{-2n}\log\frac{1}{|z|}dV(z)\right)\om(r)r^{2n-1}dr\nonumber\\
=&2np^2\int_0^1 \left(\int_0^1 \int_{\SS} |\Re f(rs\eta)|^2 |f(rs\eta)|^{p-2}s^{-1}\log\frac{1}{s}d\sigma(\eta)ds\right)\om(r)r^{2n-1}dr\nonumber\\
=&2np^2\int_0^1 \left(\int_0^r \int_{\SS} |\Re f(t\eta)|^2 |f(t\eta)|^{p-2}t^{-1}\log\frac{r}{t}d\sigma(\eta)dt\right)\om(r)r^{2n-1}dr\nonumber\\
=&2np^2\int_0^1 \int_{\SS}\left(\int_t^1  r^{2n-1}\log\frac{r}{t}\om(r)dr\right)|\Re f(t\eta)|^2 |f(t\eta)|^{p-2}t^{-1}d\sigma(\eta)dt\nonumber\\
=&p^2  \int_{\BB}\frac{|\Re f(z)|^2 |f(z)|^{p-2}}{|z|^{2n}}  \om^{n*}(z)dV(z)  \nonumber\\
\gtrsim &p^2  \int_{\BB}{|\Re f(z)|^2 |f(z)|^{p-2}}  \om^{*}(z)dV(z).\label{0624-2}
\end{align}
Hence (\ref{0112-1}) holds.

Suppose $p\geq 2$. By Theorem 4.17 in \cite{Zk2005}, we have
$$|f(z)|\leq \frac{\|f\|_{H^p}}{(1-|z|^2)^\frac{n}{p}}, \mbox{ for all } p>0.$$
So, for all $|z|<\frac{3}{4}$, we have
$$
|f(z)|^p=\left|f_{\frac{4}{5}}(\frac{5}{4}z)\right|^p
\lesssim\|f_{\frac{4}{5}}\|_{H_p}^p \leq \frac{\int_{\frac{4}{5}}^{1}\|f_r\|_{H^p}^pr^{2n-1}\om(r)dr}{\int_{\frac{4}{5}}^1 r^{2n-1}\om(r)dr}
\lesssim \|f\|_{A_\om^p}^p.
$$
Let $z=(z_1,z_2,\cdots,z_n)\in\BB$.
 By Cauchy's Formula, we obtain
\begin{align*}
\left|\frac{\partial f}{\partial z_i}(z)\right|
 \lesssim \|f\|_{A_\om^p},
\end{align*}
when $|z|< \frac{1}{2}$ and $i=1,2,\cdots,n$.
So, we have
$$|\Re f(z)|=|\langle \nabla f(z),\overline{z}\rangle|\lesssim |z|\|f\|_{A_\om^p}, \mbox{ when } |z|\leq \frac{1}{2}.$$
Here $\nabla f(z)=\Big(\frac{\partial f }{\partial z_1 }(z),\cdot\cdot\cdot,\frac{\partial f }{\partial z_n}(z)\Big)$.
For all $\tau<\frac{1}{2}$, we have
\begin{align*}
\|f\|_{A_\om^p}^p
\lesssim& \int_{\frac{1}{2}}^1 \|f_r\|_{H^p}^pr^{2n-1}\om(r)dr\\
\approx&\int_\frac{1}{2}^1 \left(\int_0^\tau+\int_\tau^1\right)
 \left(\int_{\SS} |\Re f(rs\eta)|^2|f(rs\eta)|^{p-2} s^{-1}\log\frac{1}{s}d\sigma(\eta)ds\right)\om(r)r^{2n-1}dr\\
\lesssim &\int_\frac{1}{2}^1 \left(\int_\tau^1 \int_{\SS} |\Re f(rs\eta)|^2|f(rs\eta)|^{p-2} s^{-1}\log\frac{1}{s}d\sigma(\eta)ds\right)\om(r)r^{2n-1}dr\\
&+\|f\|_{A_\om^p}^p  \int_0^\tau s\log\frac{1}{s}ds .
\end{align*}
Since $\lim\limits_{\tau\to 0}\int_0^\tau s\log\frac{1}{s}ds =0$, we can choose a fixed $\tau\in (0,\frac{1}{2})$ such that
\begin{align*}
\|f\|_{A_\om^p}^p
\lesssim \int_\frac{1}{2}^1 \left(\int_\tau^1 \int_{\SS} |\Re f(rs\eta)|^2|f(rs\eta)|^{p-2} s^{-1}\log\frac{1}{s}d\sigma(\eta)ds\right)\om(r)r^{2n-1}dr.
\end{align*}
By Fubini's Theorem, we have
\begin{align}
\|f\|_{A_\om^p}^p
\lesssim&\int_\frac{\tau}{2}^1 \left(\int_{r\tau}^r \int_{\SS} |\Re f(t\eta)|^2|f(t\eta)|^{p-2} t^{-1}\log\frac{r}{t}d\sigma(\eta)dt\right)\om(r)r^{2n-1}dr  \nonumber\\
\leq&\int_\frac{\tau^2}{2}^1 \int_{\SS}\left(\int_t^1  r^{2n-1}\log\frac{r}{t}\om(r)dr\right)|\Re f(t\eta)|^2|f(t\eta)|^{p-2} t^{-1}d\sigma(\eta)dt   \nonumber\\
\approx&  \int_{\BB\backslash \frac{\tau^2}{2}\BB}\frac{|\Re f(z)|^2|f(z)|^{p-2} }{|z|^{2n}}  \om^{n*}(z)dV(z)  \label{0624-3}\\
\lesssim & \int_{\BB}|\Re f(z)|^2|f(z)|^{p-2}   \om^{*}(z)dV(z).  \nonumber
\end{align}
So,  we get (\ref{0112-4}).

Since $\tau\in(0,1)$ is fixed, after a calculation, we have
\begin{itemize}
  \item For all $z\in\BB$, $\om^{*}(z)\geq  \om^{n*}(z)\gtrsim (1-|z|)\hat{\om}(z)$.
  \item For all $|z|>\frac{\tau^2}{2}$, $\om^{*}(z)\approx \om^{n*}(z)\approx (1-|z|)\hat{\om}(z)$.
\end{itemize}
So, using (\ref{0624-2}) and (\ref{0624-3}), we obtain (\ref{0624-1}).

By Theorem B in \cite{Pj2016jfa}, if $f(0)=0$, we have
\begin{align}\label{0121-1}
\|f\|_{H^p}^p\approx \int_{\SS}\left( \int_{\Gamma_\zeta} |\Re f(z)|^2(1-|z|^2)^{1-n}dV(z)   \right)^{\frac{p}{2}}  d\sigma(\zeta).
\end{align}
By Fubini's Theorem, we have
\begin{align*}
\|f\|_{A_\om^p}^p
=& 2n\int_0^1 \|f_r\|_{H^p}^p\om(r)r^{2n-1}dr  \\
\approx &  2n\int_0^1 \left( \int_{\SS}\left( \int_{\Gamma_\zeta} |\Re f_r(z)|^2(1-|z|^2)^{1-n}dV(z)   \right)^{\frac{p}{2}}  d\sigma(\zeta)   \right)\om(r)r^{2n-1}dr  \\
=& \int_{\BB}\left(\int_{\Gamma_u} |\Re f(\xi)|^2\left(1-\frac{|\xi|^2}{|u|^2}\right)^{1-n} dV(\xi)\right)^{\frac{p}{2}} \om(u) |u|^{-np}dV(u)\\
\geq &\int_{\BB}\left(\int_{\Gamma_u} |\Re f(\xi)|^2\left(1-\frac{|\xi|^2}{|u|^2}\right)^{1-n} dV(\xi)\right)^{\frac{p}{2}} \om(u)dV(u).
\end{align*}
Similarly, by using the monotonicity of  $\|f_r\|_{H^p}$,
 we have
\begin{align}
\|f\|_{A_\om^p}^p
=& 2n\int_0^1 \|f_r\|_{H^p}^p\om(r)r^{2n-1}dr  \nonumber\\
\lesssim& 2n\int_{\frac{1}{2}}^1\|f_r\|_{H^p}^p\om(r)r^{2n-1}dr \nonumber\\
\approx& \int_{\BB\backslash \frac{1}{2}\BB}\left(\int_{\Gamma_u} |\Re f(\xi)|^2\left(1-\frac{|\xi|^2}{|u|^2}\right)^{1-n} dV(\xi)\right)^{\frac{p}{2}} \om(u) |u|^{-np}dV(u)\label{0624-4}\\
\lesssim &\int_{\BB}\left(\int_{\Gamma_u} |\Re f(\xi)|^2\left(1-\frac{|\xi|^2}{|u|^2}\right)^{1-n} dV(\xi)\right)^{\frac{p}{2}} \om(u)dV(u).\nonumber
\end{align}
Then,  (\ref{0112-3}) holds.
The proof is complete.
\end{proof}

For any $f\in H(\BB)$ and $u\in\overline{\BB}\backslash\{0\}$, let
$$N(f)(u)=\sup_{z\in \Gamma_u}|f(z)|.$$
Then we have the following theorem.
\begin{Theorem}\label{0117-1} Let $0<p<\infty$ and $\om$ be a radial weight. Then for all $f\in H(\BB)$,
$$\|f\|_{A_\om^p}^p \leq \|N(f)\|_{L_\om^p}^p\lesssim \|f\|_{A_\om^p}^p.$$
\end{Theorem}
\begin{proof}
For any $u\in\BB\backslash\{0\}$, let $r=|u|$ and $\xi=\frac{u}{|u|}$. Then
\begin{align*}
\Gamma_u
&=\left\{z\in\BB:\left|1-\langle z,\frac{u}{|u|^2}\rangle\right| \leq \frac{\alpha}{2}\left( 1- \frac{|z|^2}{|u|^2}\right) \right\} \\
&=\left\{\frac{z}{r}\in\BB:\left|1-\langle \frac{z}{r},\xi\rangle\right| \leq \frac{\alpha}{2}\left( 1- \left|\frac{z}{r}\right|^2\right) \right\},
\end{align*}
and
$$N(f)(u)=\sup\limits_{\frac{z}{r}\in\Gamma_\xi}\left\{|f(r\frac{z}{r})|\right\}=N(f_r)(\xi).$$
By Theorem A in \cite{Pj2016jfa}, we have
$\|N(f)\|_{L^p(\SS)}^p\lesssim \|f\|_{H^p}^p.$
 Therefore,
\begin{align*}
\|N(f)\|_{L_\om^p}^p
&=2n\int_0^1 \|N(f_r)\|_{L^p(\SS)}^p r^{2n-1}\om(r)dr\\
&\lesssim 2n\int_0^1 \|f_r\|_{H^p}^p r^{2n-1}\om(r)dr =\|f\|_{A_\om^p}^p.
\end{align*}
The fact that $\|f\|_{A_\om^p}^p\leq \|N(f)\|_{L_\om^p}^p$ is obvious. The proof is complete.
\end{proof}

\section{Volterra integral operator  from $A_\om^p$ to $A_\om^q$}

In this section, we will describe the boundedness and compactness of $T_g:A_\om^p\to A_\om^q$.
For this purpose, we first introduce some new function spaces.

Let $0<p\leq q<\infty$, $g\in H(\BB)$ and $\om\in\hD$. We say that $g$ belongs to $\mathcal{C}^{q,p}(\om^*)$, if the measure
$|\Re g(z)|^2\om^*(z)dV(z)$ is a $q-$Carleson measure for $A_\om^p$.
 $g\in \mathcal{C}_0^{q,p}(\om^*)$ if $|\Re g(z)|^2\om^*(z)dV(z)$ is a vanishing $q-$Carleson measure for $A_\om^p$.
If $0<p\leq q <\infty$,  Theorem  \ref{0109-7} shows that  $\mathcal{C}^{q,p}(\om^*)$  depends only on $\frac{q}{p}$.
Consequently, for $0<p\leq q<\infty$, we will write $\mathcal{C}^{\kappa}(\om^*)$ instead of $\mathcal{C}^{q,p}(\om^*)$  where $\kappa=\frac{q}{p}$.
Similarly, we can define $\mathcal{C}_0^\kappa(\om^*)$. Thus, if $\kappa\geq 1$, $\mathcal{C}^\kappa(\om^*)$ consists of those $g\in H(\BB)$ such that
$$\|g\|_{\mathcal{C}^\kappa(\om^*)}=|g(0)|+\sup_{a\in\BB}\frac{\int_{S_a}|\Re g(z)|^2\om^*(z)dV(z)}{(\om(S_a))^\kappa}<\infty.$$

Before state and prove  the main results in this section, we state some lemmas which will be used.
For brief, if $r\in (0,1)$, let $S_r$ denote any Carleson block $S_a$ with $|a|=r$.

\begin{Lemma}\label{0114-3}
Let $0<p,q<\infty$, $g\in H(\BB)$  and $\om\in\hD$.
\begin{enumerate}[(i)]
  \item If $T_g:A_\om^p\to A_\om^q$ is bounded, then
  $$M_\infty(r, \Re g)\lesssim \frac{\om^{\frac{1}{p}-\frac{1}{q}}(S_r)}{1-r},\,\,0<r<1.$$
  \item If $T_g:A_\om^p\to A_\om^q$ is compact, then
  $$M_\infty(r, \Re g)=o\left( \frac{\om^{\frac{1}{p}-\frac{1}{q}}(S_r)}{1-r}\right),\,\,r\to 1.$$
\end{enumerate}
\end{Lemma}

\begin{proof} Assume $T_g:A_\om^p\to A_\om^q$ is bounded. Let
$$f_{a,p}(z)=\frac{F_{a,p}(z)}{(\om(S_a))^\frac{1}{p}} =\left(\frac{1-|a|^2}{1-\langle z,a \rangle}\right)^\frac{n+\gamma}{p}\frac{1}{(\om(S_a))^\frac{1}{p}}$$
for some $\gamma$ which is large enough such that Lemma \ref{1208-4} holds.
For all $\frac{1}{2}<r<1$ and $h\in A_\om^q$, we have
\begr \|h\|_{A_\om^q}^q
&\geq & \int_{\BB\backslash r\BB} |h(z)|^q\om(z)dV(z)\nonumber\\
&\geq& 2nM_q^q(r,h)\int_r^1 r^{2n-1}\om(r)dr\approx \hat{\om}(r)M_q^q(r,h).\nonumber
\endr
Then, when $\frac{1}{2}<r<1$, for all $a\in\BB$, by Lemma \ref{1208-4}, we have
\begr\label{0113-3}
M_q^q(r,T_g f_{a,p})
 \lesssim  \frac{\|T_g f_{a,p}\|_{A_\om^q}^q}{\hat{\om}(r)}
 \lesssim
\frac{\|T_g\|_{A_\om^p\to A_\om^q}^q\|f_{a,p}\|_{A_\om^p}^q}{\hat{\om}(r)}
\lesssim \frac{1}{\hat{\om}(r)}.
\endr
The following facts are well know, that are
$$|f(z)|\leq \frac{\|f\|_{H^q}}{(1-|z|^2)^\frac{n}{q}},\,\,\mbox{and}\,\,|\Re f(z)|\lesssim \frac{\|f\|_{H^q}}{(1-|z|^2)^{\frac{n}{q}+1}}.$$
Letting $f_r(z)=f(rz)$ for all $0<r<1$, when $|a|>\frac{1}{2}$, by (\ref{0113-3}), we have
\begin{align*}
\frac{|\Re g(a)|}{(\om(S_a))^{\frac{1}{p}}}
&=|\Re(T_gf_{a,p})(a)|
= |\Re((T_gf_{a,p})_{\frac{|a|+1}{2}})(\frac{2a}{1+|a|})| \\
&\lesssim \frac{\|(T_gf_{a,p})_{\frac{|a|+1}{2}}\|_{H^q}}{(1-|a|)^{\frac{n}{q}+1}}
=\frac{M_q(\frac{|a|+1}{2},T_g f_{a,p})}{(1-|a|)^{\frac{n}{q}+1}} \\
&\lesssim \frac{1}{(1-|a|)^{\frac{n}{q}+1}\hat{\om}^\frac{1}{q}(\frac{|a|+1}{2})}
\approx  \frac{1}{(1-|a|)^{\frac{n}{q}+1}\hat{\om}^\frac{1}{q}(a)}.
\end{align*}
By Lemma \ref{1210-3}, we have
$$|\Re g(a)|\lesssim \frac{\om^{\frac{1}{p}-\frac{1}{q}}(S_a)}{1-|a|},$$
which implies the desired result.

$(ii)$ Assume that $T_g:A_\om^p\to A_\om^q$ is compact.
By Lemma \ref{1208-4}, $\{f_{a,p}\}$ is bounded and converges to 0 uniformly on compact subset of $\BB$ as $|a|\to 1$.
By Lemma \ref{1210-2}, $$\lim\limits_{|a|\to 1}\|T_g f_{a,p}\|_{A_\om^q}=0.$$
By (\ref{0113-3}), for any given $\varepsilon>0$, there exists a $r_0\in(0,1)$, such that when $|a|>r_0$,
\begin{align*}
M_q^q(r,T_g f_{a,p})
\lesssim \frac{\varepsilon}{\hat{\om}(r)}.
\end{align*}
Then by repeating the proof of $(i)$, we can prove $(ii)$. The proof is complete.
\end{proof}

\begin{Lemma}\label{0114-4}
Let $0<\kappa<\infty$, $\om\in\hD$ and $g\in H(\BB)$. Then the following statements hold.
\begin{enumerate}[(i)]
  \item $g\in \mathcal{C}^{2\kappa+1}(\om^*)$ if and only if
 \begin{align}\label{0114-2}
 M_\infty(r,\Re g)\lesssim \frac{\om^\kappa(S_r)}{1-r},\,\,0<r<1.
 \end{align}
  \item $g\in \mathcal{C}_0^{2\kappa+1}(\om^*)$ if and only if
  $$M_\infty(r,\Re g)=o\left( \frac{\om^\kappa(S_r)}{1-r}\right),\,\,r\to 1.$$
\end{enumerate}
\end{Lemma}

\begin{proof} Let $r_0>0$ be fixed and  $D(a,r_0)$ be the Bergman metric ball at $a$ with radius $r_0$.
By Lemma 2.20 in \cite{Zk2005}, there exists $B=B(r_0)>1$ such that, for all $z\in D(a,r_0)$,
$$B^{-1}< \frac{1-|a|}{1-|z|}<B, \,\,\mbox{ and }\,\,B^{-1}<\frac{1-|a|}{|1-\langle a,z\rangle|}<B.$$
When $|a|> \max\{\frac{B-1}{B},\frac{2B}{2B+1}\}=\frac{2B}{2B+1}$, let
$$\beta_3(a)=\frac{a-(2B+1)(1-|a|)a}{|a|}\in \BB.$$
Then, $1-|\beta_3(a)|=(2B+1)(1-|a|)$. For all $z\in D(a,r_0)$, we have
$$
|z|>1+B|a|-B>|\beta_3(a)|
$$
and
\begin{align*}
\left|1-\langle \frac{\beta_3(a)}{|\beta_3(a)|},\frac{z}{|z|}  \rangle\right|
&\leq \left|1-\langle a,z  \rangle\right|    +    \left|\langle a,z\rangle-\langle a,\frac{z}{|z|}  \rangle\right|
        +     \left|\langle a,\frac{z}{|z|}\rangle- \langle \frac{a}{|a|},\frac{z}{|z|}  \rangle\right|\\
&\leq B(1-|a|)+|a|(1-|z|) +(1-|a|)\\
&\leq (2B+1)(1-|a|)=1-|\beta_3(a)|.
\end{align*}
Therefore, $D(a,r_0)\subset S_{\beta_{3}(a)}$ for all $|a|>\max\{\frac{B-1}{B},\frac{2B}{2B+1}\}$.

Assume that $g\in \mathcal{C}^{2\kappa+1}(\om^*)$.
It is enough to prove (\ref{0114-2}) holds for $|a|> \frac{2B}{2B+1}$.
By Lemma \ref{1210-3}, we have $\om^*\in\R$ and
$$\om^*(t)\approx \om^*(s),\,\,\mbox{ if }\,\, 1-t\approx 1-s\,\mbox{ and }\,s,t\in(\varepsilon,1),$$
here $\varepsilon$ is any fixed positive number in (0,1).

When $|a|>\max\{\frac{B-1}{B},\frac{2B}{2B+1}\}$, by   Lemma 2.24 in \cite{Zk2005},  Lemma \ref{1210-3} and Theorem \ref{0109-7},
there is a $C=C(\kappa,r_0,\om)$ such that
\begr
(1-|a|^2)^{n+1}\om^*(a)|\Re g(a)|^2
&\leq & C\om^*(a)\int_{D(a,r_0)}|\Re g(z)|^2dV(z)\nonumber\\
& \leq & C \int_{D(a,r_0)}|\Re g(z)|^2\om^*(z)dV(z)\nonumber\\
&\leq &C\int_{S_{\beta_3(a)}}|\Re  g(z)|^2\om^*(z)dV(z)\nonumber\\
&\leq &C\|g\|_{\mathcal{C}^{2\kappa+1}(\om^*)} \om^{2\kappa+1}(S_{\beta_3(a)})\nonumber\\
&\leq  &C \|g\|_{\mathcal{C}^{2\kappa+1}(\om^*)}  \om^{2\kappa+1}(S_a).\nonumber
\endr
Hence, there exists $C=C(k,r_0,\om)$, such that
$$|\Re g(a)|\leq C\|g\|_{\mathcal{C}^{2\kappa+1}(\om^*)} \frac{\om^\kappa(S_a)}{1-|a|},\,\,\mbox{ when }\,\,|a|\geq \frac{2B}{2B+1}.$$
Therefore, (\ref{0114-2}) holds.

Conversely, suppose that
$$M:=\sup_{0< |a|=r<1} \frac{(1-r)M_\infty(r,\Re g)}{\om^\kappa(S_a)}<\infty.$$
By Lemmas \ref{1212-1} and  \ref{1210-3}, we have
\begin{align*}
\int_{S_a}|\Re g(z)|^2\om^*(z)dV(z)&=2n\int_{|a|}^1 \int_{Q_a} |\Re g(r\xi)|^2\om^*(r)r^{2n-1}d\sigma(\xi)dr\\
&\leq M^2\int_{|a|}^1\int_{Q_a} \frac{\om^{2\kappa}(S_r)}{(1-r)^2} \om^*(r)r^{2n-1}d\sigma(\xi)dr\\
&\approx  M^2(1-|a|)^n\int_{|a|}^1 \frac{\om^{2\kappa}(S_r)}{(1-r)^2} \om^*(r)dr\\
&\lesssim M^2 \om^{2\kappa+1}(S_a).
\end{align*}
It follows that $g\in \mathcal{C}^{2\kappa+1}(\om^*)$.

The assertion $(ii)$ can be proved by  modifying the  above proof in a standard way
and we omit the detail.  The proof is complete.
\end{proof}

\begin{Theorem}\label{0121-4}
Let $0< p\leq q<\infty$, $\om\in\hD$, $\kappa=\frac{1}{p}-\frac{1}{q}$ and $g\in H(\BB)$.
\begin{enumerate}
  \item [(i)] If $n\kappa\geq 1$, then $T_g:A_\om^p\to A_\om^q$ is bounded if and only if $g$ is constant.
  \item [(ii)]If $0<n\kappa<1$, then the following conditions are equivalent:
  \begin{enumerate}
    \item [(iia)] $T_g:A_\om^p\to A_\om^q$ is bounded;
    \item [(iib)] $M_\infty(r,\Re g)\lesssim \frac{\om^{\kappa}(S_r)}{1-r}$
    \item [(iic)] $g\in \mathcal{C}^{2\kappa+1}(\om^*)$.
  \end{enumerate}
  \item [(iii)] The following conditions are equivalent.
  \begin{enumerate}
    \item [(iiia)] $T_g:A_\om^p\to A_\om^p$ is bounded;
    \item [(iiib)] $g\in \mathcal{C}^1(\om^*)$.
  \end{enumerate}
\end{enumerate}
\end{Theorem}

\begin{proof} By Lemmas \ref{1210-3}, \ref{0114-3} and \ref{0114-4},  we see that $(i)$ holds, and $(iia)\Rightarrow(iib)\Leftrightarrow (iic)$.
Let $d\mu_g(z)=|\Re g(z)|^2\om^*(z)dV(z)$. First, we  prove the statement $(ii)$.

Suppose that $(iic)$ holds and  $q=2$. Then  $d\mu_g$ is a 2-Carleson measure for $A_\om^p$.
By using (\ref{0112-4}) and Theorem \ref{0109-7}, we have
$$\|T_g f\|_{A_\om^2}^2\approx\int_{\BB}|f(z)|^2|\Re g(z)|^2\om^*(z)dV(z)\lesssim \|f\|_{A_\om^p}^2.$$
So, $T_g:A_\om^p\to A_\om^2$ is bounded.

Assume  $f\in H^\infty$.
By Lemmas \ref{1210-3} and  \ref{0114-3},  we get
\begin{align*}
\sup\limits_{0< t\leq \frac{1}{2},z\in\BB}\frac{|\Re g(tz)|}{t}<\infty,\,\,
\left|\int_0^\frac{1}{2}f(tz)\Re g(tz)\frac{dt}{t} \right|\leq \|f\|_{H^\infty}\int_0^\frac{1}{2}|\Re g(tz)|\frac{dt}{t}<\infty,
\end{align*}
and
\begin{align*}
\left|\int_\frac{1}{2}^1 f(tz)\Re g(tz)\frac{dt}{t}\right|
\leq \|f\|_{H^\infty}\int_{\frac{1}{2}}^1 (1-t|z|)^{n\kappa-1}\hat{\om}^\kappa(t|z|)dt<\infty.
\end{align*}
That is to say $T_g f\in H^\infty$.

Suppose that $(iic)$ holds and $q>2$. Let $\beta=\frac{(2\kappa+1)q}{2kq+2}$ and $\beta^\p=\frac{(2\kappa+1)q}{q-2} $.
By (\ref{0112-4}) and H\"{o}lder's inequality, we have
\begin{align*}
\|T_gf\|_{A_\om^q}^q
&\approx \int_{\BB} |f(z)|^2|T_gf(z)|^{q-2}|\Re g(z)|^2\om^*(z)dV(z)   \\
&\leq \left(\int_{\BB}|f(z)|^{2\beta}d\mu_g(z)\right)^\frac{1}{\beta}
    \left(\int_{\BB}|T_g f(z)|^{(q-2)\beta^\p}d\mu_g(z)\right)^\frac{1}{\beta^\p}  \\
&\lesssim  \|f\|_{A_{\om}^\frac{2\beta}{2\kappa+1}}^2          \|T_gf\|_{A_{\om}^\frac{(q-2)\beta^\p}{2\kappa+1}}^{q-2}
= \|f\|_{A_{\om}^p}^2         \|T_gf\|_{A_{\om}^q}^{q-2}.
\end{align*}
Therefore, when $q>2$,
\begin{align}\label{0116-3}
\|T_gf\|_{A_\om^q}\lesssim \|f\|_{A_{\om}^p}, \,\,\mbox{ for all }\,\,f\in H^\infty.
\end{align}
For all $f\in A_\om^p$, we have $\lim\limits_{r\to 1}\|f-f_r\|_{A_\om^p}=0$.
Then $\{f_r\}$ is a Cauchy's sequence in $A_\om^p$.
By (\ref{0116-3}), there is a $F\in A_\om^q$, such that $\lim\limits_{r\to 1}\|T_gf_r-F\|_{A_\om^q}=0$.
So, for all $z\in\BB$, we have $\lim\limits_{r\to 1}|T_gf_r(z)-F(z)|=0$.

Meanwhile, for any fixed $z\in \BB$, we have
\begin{align*}
\lim_{r\to 1}|T_gf(z)-T_gf_r(z)|
&\leq  \left(\sup_{t\in[0,1]}\left|\frac{\Re g(tz)}{t}\right|\right)
\left(\lim_{r\to 1}\sup_{t\in[0,1]}|\Re f(tz)-\Re f(rtz)|\right)=0.
\end{align*}
Therefore, $F=T_g f$ and
$$\|T_g f\|_{A_\om^q}=\lim_{r\to 1}\|T_g f_r\|_{A_\om^q}\lesssim  \limsup_{r\to 1}\| f_r\|_{A_\om^p}=\|f\|_{A_\om^p}.$$
So, $(iic)$ deduce $(iia)$ when $q>2$.

Suppose that $(iic)$ holds and $q<2$. Let $\tau=\frac{(2-q)p}{q}$. By (\ref{0112-3}), H\"older's inequality and Theorem \ref{0117-1},  we have
\begin{align}
\|T_gf\|_{A_\om^q}^q
&\approx\int_{\BB}\left(\int_{\Gamma_u}|f(\xi)\Re g(\xi)|^2\left(1-\frac{|\xi|^2}{|u|^2}\right)^{1-n}dV(\xi)\right)^\frac{q}{2}\om(u)dV(u)   \nonumber\\
&\leq \int_{\BB}|N(f)(u)|^\frac{\tau q}{2}
    \left(\int_{\Gamma_u}|f(\xi)|^{2-\tau}|\Re g(\xi)|^2\left(1-\frac{|\xi|^2}{|u|^2}\right)^{1-n}dV(\xi)\right)^\frac{q}{2}\om(u)dV(u)
\nonumber \\
&\leq \|N(f)\|_{L_\om^p}^\frac{(2-q)p}{2} J_1^{\frac{q}{2}}\lesssim \|f\|_{A_\om^p}^\frac{(2-q)p}{2}J_1^{\frac{q}{2}},\label{0122-3}
\end{align}
where
\begin{align}\label{0117-4}
J_1=\int_{\BB}\int_{\Gamma_u}|f(\xi)|^{2-\tau}|\Re g(\xi)|^2\left(1-\frac{|\xi|^2}{|u|^2}\right)^{1-n}dV(\xi)\om(u)dV(u).
\end{align}
By   Fubini's Theorem, we have
\begin{align}
J_1&=2n\int_0^1 r^{2n-1}\om(r)dr\int_{\SS}d\sigma(\eta)
      \int_{\Gamma_{r\eta}}|f(\xi)|^{2-\tau}|\Re g(\xi)|^2\left(1-\frac{|\xi|^2}{|r|^2}\right)^{1-n}dV(\xi)  \nonumber \\
&=2n\int_0^1 r^{4n-1}\om(r)dr\int_{\SS}d\sigma(\eta)
      \int_{\Gamma_{\eta}}|f(rz)|^{2-\tau}|\Re g(rz)|^2(1-|z|^2)^{1-n}dV(z)   \nonumber\\
&=2n \int_0^1 r^{4n-1}\om(r)dr  \int_{\BB}  |f(rz)|^{2-\tau}|\Re g(rz)|^2(1-|z|^2)^{1-n}\sigma(E_z)dV(z).  \nonumber
\end{align}
Here, $E_z=\{\eta\in\SS:\left|1-\langle z,\eta\rangle\right|<\frac{\alpha}{2}(1-|z|^2)\}$.
By the definition of $\Gamma_\eta$, we have $\alpha>2$.
If $z\neq 0$, we have
\begin{align}\label{0625-1}
Q\left(\frac{z}{|z|},\sqrt{\frac{(\alpha-2)(1-|z|)}{2}}\right)\subset
E_z\subset Q\left(\frac{z}{|z|},\sqrt{(\alpha+1)(1-|z|)}\right),
\end{align}
by the facts of
\begin{align*}
|1-\langle z,\zeta \rangle|\leq \left|1-\langle \frac{z}{|z|},\zeta \rangle\right|+\left|\langle \frac{z}{|z|},\zeta \rangle -\langle z,\zeta \rangle\right|
<\frac{\alpha}{2}(1-|z|^2),
\end{align*}
and
\begin{align*}
\left|1-\langle \frac{z}{|z|},\eta \rangle\right|
&\leq \left|1-\langle z,\eta \rangle\right|+\left|\langle z,\eta \rangle-\langle \frac{z}{|z|},\eta \rangle \right|
<(\alpha+1)(1-|z|),
\end{align*}
when $\zeta\in Q\left(\frac{z}{|z|},\sqrt{\frac{(\alpha-2)(1-|z|)}{2}}\right)$ and $\eta\in E_z$, respectively.
By Lemma \ref{1212-1}, we have $\sigma(E_z)\approx (1-|z|)^n$.

Since $ \int_t^1 r^{2n-1}\om(r)(1-\frac{t}{r})dr\approx  (1-t) \hat{\om}(t)$, by   Fubini's Theorem, (\ref{0112-4}) and (\ref{0624-1}), we have
\begin{align}
J_1&\approx \int_0^1 r^{4n-1}\om(r)dr  \int_{\BB}  |f(rz)|^{2-\tau}|\Re g(rz)|^2(1-|z|^2)dV(z)  \nonumber\\
&\approx \int_0^1 r^{4n-1}\om(r)dr  \int_0^1 s^{2n-1}ds\int_{\SS}  |f(rs\eta)|^{2-\tau}|\Re g(rs\eta)|^2(1-s^2)d\sigma(\eta)  \nonumber\\
&\approx \int_0^1 r^{2n-1}\om(r)dr  \int_0^r t^{2n-1}dt\int_{\SS}  |f(t\eta)|^{2-\tau}|\Re g(t\eta)|^2(1-\frac{t}{r})d\sigma(\eta)\nonumber\\
&\approx   \int_0^1 t^{2n-1}dt\int_{\SS}  |f(t\eta)|^{2-\tau}|\Re g(t\eta)|^2d\sigma(\eta)\int_t^1 r^{2n-1}\om(r)(1-\frac{t}{r})dr  \nonumber\\
&\approx \int_{\BB}|f(z)|^{2-\tau}|\Re g(z)|^2 (1-|z|)\hat{\om}(z)dV(z)  \nonumber\\
&\approx \int_{\BB}|f(z)|^{2-\tau}|\Re g(z)|^2 \om^*(z)dV(z)  \label{0117-5}\\
&\lesssim \|f\|_{A_\om^{\frac{2-\tau}{2\kappa+1}}}^{2-\tau}=\|f\|_{A_\om^p}^{2-\tau}.\nonumber
\end{align}
So, $\|T_g f\|_{A_\om^q}\lesssim \|f\|_{A_\om^p}$.  Then we finish the proof of assertion $(ii)$.

Next we prove the assertion $(iii)$.  When $p=2$, $(iiia)$$\Leftrightarrow$$(iiib)$ is obvious.
When $p<2$, by   the the  proof of $(iic)$$\Rightarrow$$(iia)$ when $q <2$, we obtain $(iiib)$$\Rightarrow$$(iiia)$.

Suppose $2<p\leq 4$ and $(iiib)$ holds. For all $f\in H^\infty$, by (\ref{0112-4}), we have
\begin{align*}
\|T_gf\|_{A_\om^p}^p
&\approx \int_{\BB}|T_g f(z)|^{p-2}|f(z)\Re g(z)|^2  \om^*(z) dV(z)  \\
&\leq \|f\|_{H^\infty}^2 \int_{\BB}|T_g f(z)|^{p-2}|\Re g(z)|^2  \om^*(z) dV(z) .
\end{align*}
By $g\in \mathcal{C}^1(\om^*)$, we see that  $T_g$ is bounded on $A_\om^{p-2}$. Since $f\in H^\infty\subset A_\om^{p-2}$, $T_g f\in A_\om^{p-2}$.
By Theorem \ref{0109-7}, we have
$$\int_{\BB}|T_g f(z)|^{p-2}|\Re g(z)|^2  \om^*(z) dV(z)\lesssim \|T_g f\|_{A_\om^{p-2}}^{p-2}\lesssim \|f\|_{A_\om^{p-2}}^{p-2}.$$
Similar to the proof of $(iic)\Rightarrow(iia)$, we obtain $(iiib)$$\Rightarrow$$(iiia)$ when $2<p\leq 4$.
Using mathematical induction,  we have $(iiib)$$\Rightarrow$$(iiia)$ when $p>2$.

So, it remains to show that $(iiia)$$\Rightarrow$$(iiib)$ when $p\neq 2$.

Suppose $p>2$ and $(iiia)$ holds.
By the calculations from  (\ref{0117-4}) to (\ref{0117-5}), H\"older's inequality, Theorem \ref{0117-1} and (\ref{0112-3}), we have
\begin{align}
&\int_{\BB}|f(z)|^p|\Re g(z)|^2\om^*(z)d\nu(z)  \nonumber \\
\approx&\int_{\BB}\int_{\Gamma_u}|f(\xi)|^p |\Re g(\xi)|^2\left(1-\frac{|\xi|^2}{|u|^2}\right)^{1-n}dV(\xi)\om(u)dV(u) \nonumber\\
\leq &\int_{\BB} |N(f)(u)|^{p-2}\int_{\Gamma_u}|f(\xi)|^2 |\Re g(\xi)|^2\left(1-\frac{|\xi|^2}{|u|^2}\right)^{1-n}dV(\xi)\om(u)dV(u) \nonumber\\
\leq& \|N(f)\|_{L_\om^p}^{p-2}
       \left(\int_{\BB} \left(\int_{\Gamma_u}|f(\xi)|^2 |\Re g(\xi)|^2\left(1-\frac{|\xi|^2}{|u|^2}\right)^{1-n}dV(\xi)\right)^{\frac{p}{2}}\om(u)dV(u)\right)^\frac{2}{p} \nonumber\\
\approx& \|f\|_{A_\om^p}^{p-2}\|T_g f\|_{A_\om^p}^2  \lesssim \|f\|_{A_\om^p}^p.  \label{0122-4}
\end{align}
So, $|\Re g(z)|^2\om^*(z)dV(z)$ is a $p-$Carleson measure for $A_\om^p$, and thus $g\in \mathcal{C}^1(\om^*)$.

Suppose $0<p<2$ and $(iiia)$ holds.
Recall that $d\mu_g(z)=|\Re g(z)|^2\om^*(z)dV(z)$.
Then by Lemma \ref{0114-3} and its proof we get
\begin{align}\label{0121-2}
g\in\B\,\,\mbox{ and }\,\,\|g-g(0)\|_{\B}\lesssim\|T_g\|.
\end{align}
Here, $\B$ is the Bloch space on the unit ball and $\|T_g\|$ is $\|T_g\|_{A_\om^p\to A_\om^p}$.
Let $F_{a,p}$  be defined as (\ref{0119-2}) for some $\gamma$ large enough.
Let $1<\tau_1,\tau_2<\infty$ such that $\frac{\tau_2}{\tau_1}=\frac{p}{2}<1$, and let $\tau_1^\p,\tau_2^\p$ be the conjugate indexes of $\tau_1, \tau_2$ respectively.

By Lemma \ref{1210-3}, Proposition \ref{0110-1}, H\"older's inequality, Fubini's Theorem and (\ref{0112-3}),  for any $a\in\BB$ with $|a|\geq \frac{1}{2}$,  we have
\begin{align}
\mu_g(S_a)
\approx& \int_{S_a} |F_{a,p}(z)|^2|\Re g(z)|^2 (1-|z|^2)^{1-n}\int_{T_z}\om(u)dV(u)dV(z)  \nonumber\\
\approx&\int_{\BB}\left(\int_{S_a\cap \Gamma_u}|F_{a,p}(z)|^2|\Re g(z)|^2 (1-|z|^2)^{1-n}dV(z)\right)^{\frac{1}{\tau_1}+\frac{1}{\tau_1^\p}} \om(u)dV(u)\nonumber
\end{align}
\begin{align}
 \leq & \left(  \int_{\BB}\left(\int_{S_a\cap \Gamma_u}|F_{a,p}(z)|^2|\Re g(z)|^2 (1-|z|^2)^{1-n}dV(z)\right)^{\frac{p}{2}} \om(u)dV(u)\right)^\frac{1}{\tau_2}    \nonumber \\
&\cdot \left(  \int_{\BB}\left(\int_{S_a\cap \Gamma_u}|F_{a,p}(z)|^2|\Re g(z)|^2 (1-|z|^2)^{1-n}dV(z)\right)^{\frac{\tau_2^\p}{\tau_1^\p}} \om(u)dV(u)\right)^\frac{1}{\tau_2^\p}   \nonumber\\
\leq& \|T_g F_{a,p}\|_{A_\om^p}^{\frac{p}{\tau_2}} \|J_2\|_{L_\om ^{{\tau_2^\p}/{\tau_1^\p}}}^\frac{1}{\tau_1^\p},\nonumber
\end{align}
where
$$J_2(u)=\int_{S_a\cap \Gamma_u}|F_{a,p}(z)|^2|\Re g(z)|^2 (1-|z|^2)^{1-n}dV(z).$$
Since $\frac{\tau_2^\p}{\tau_1^\p}>1$, we have $\left(\frac{\tau_2^\p}{\tau_1^\p}\right)^\p=\frac{\tau_2(\tau_1-1)}{\tau_1-\tau_2}>1$.
Let $\tau_3=\frac{\tau_2(\tau_1-1)}{\tau_1-\tau_2}$.  We have
$$\|J_2\|_{L_\om ^{{\tau_2^\p}/{\tau_1^\p}}}
=\sup_{\|h\|_{L_\om^{\tau_3}}\leq 1}\left|\int_{\BB}h(u)J_2(u)\om(u)dV(u)\right|.$$
By using Fubini's Theorem, Proposition \ref{0110-1}, Lemma \ref{1208-4}, Lemma \ref{1201-1}, Propositions \ref{0118-1}, Remark \ref{0308-1}, Lemma \ref{1210-3} and Corollary \ref{0119-3} in order, we have
\begin{align}
&\left|\int_{\BB}h(u)J_2(u)\om(u)dV(u)\right|\leq \int_{\BB}|h(u)|J_2(u)\om(u)dV(u)  \nonumber\\
\lesssim&\int_{S_a}|\Re g(z)|^2 (1-|z|^2)^{1-n}\int_{S_{z,\alpha}}|h(u)|\om(u)dV(u)dV(z)  \nonumber\\
\lesssim& \int_{S_a}|\Re g(z)|^2 M_\om(|h|)(z)\om^*(z)dV(z) =\int_{S_a} M_\om(|h|)(z)d\mu_g(z)  \nonumber\\
\leq & \left(\mu_g(S_a)\right)^\frac{\tau_1^\p}{\tau_2^\p}
\left(   \int_{S_a} \left(M_\om(|h|)(z)\right)^{\tau_3}d\mu_g(z)     \right)^{\frac{1}{\tau_3}}   \nonumber\\
\lesssim&   \left(\mu_g(S_a)\right)^\frac{\tau_1^\p}{\tau_2^\p}
\left(\sup_{a\in\BB}\frac{\mu_g(S_a)}{\om(S_a)}\right)^{\frac{1}{\tau_3}} \|h\|_{L_\om^{\tau_3}}. \nonumber
\end{align}
Then we have
\begin{align}\label{0119-4}
\mu_g(S_a)\lesssim \|T_g F_{a,p}\|_{A_\om^p}^{\frac{p}{\tau_2}}  \left(\mu_g(S_a)\right)^\frac{1}{\tau_2^\p}
\left(\sup_{a\in\BB}\frac{\mu_g(S_a)}{\om(S_a)}\right)^{\frac{1}{\tau_3\tau_1^\p}}.
\end{align}
By the process of obtaining (\ref{0119-4}), if we replace $g(z)$ by $g_r(z)=g(rz)$, we have
\begin{align}\label{0119-5}
\mu_{g_r}(S_a)\lesssim \|T_{g_r} F_{a,p}\|_{A_\om^p}^{\frac{p}{\tau_2}}  \left(\mu_{g_r}(S_a)\right)^\frac{1}{\tau_2^\p}
\left(\sup_{a\in\BB}\frac{\mu_{g_r}(S_a)}{\om(S_a)}\right)^{\frac{1}{\tau_3\tau_1^\p}}.
\end{align}
We now claim that there exists a constant $C=C(\om)>0$ such that
\begin{align}\label{0119-6}
\sup_{\frac{1}{2}<r<1}\|T_{g_r}(F_{a,p})\|_{A_\om^p}^p\lesssim C \|T_g\|^p\om(S_a), \,\,\frac{1}{2}\leq |a|<1.
\end{align}
Taking this for granted for a moment, by (\ref{0119-5}) and (\ref{0119-6}) we have
$$\sup_{|a|\geq \frac{1}{2}}\frac{\mu_{g_r}(S_{a})}{\om(S_a)}\lesssim \|T_g\|^2,\,\,\mbox{ for all }\,\,\frac{1}{2}<r<1.$$
By Fatou's Lemma,
$$\sup_{|a|\geq \frac{1}{2}}\frac{\mu_{g}(S_{a})}{\om(S_a)}\lesssim \|T_g\|^2.$$
So $g\in \mathcal{C}^1(\om^*)$.

It remains to prove (\ref{0119-6}).
For any fixed $r\in(\frac{1}{2},1)$,  when $\frac{1}{2}<|a|\leq \frac{1}{2-r}$, by triangle inequality,  we have
$$|1-\langle z,a\rangle|\leq \left|1-\langle \frac{z}{r},a\rangle\right|  + \frac{1-r}{2-r}\leq 2\left|1-\langle \frac{z}{r},a\rangle\right|, \,\, |z|\leq r.$$
When $\frac{1}{2}<|a|\leq \frac{1}{2-r}$, by (\ref{0112-3}) and (\ref{0624-4}),  we have
\begin{align*}
\|T_{g_r}F_{a,p}\|_{A_\om^p}^p
&\approx \int_{\BB\backslash \frac{1}{2}\BB}\left(\int_{\Gamma_u}|F_{a,p}(\eta)|^2|\Re g_r(\eta)|^2\left(1-\frac{|\eta|^2}{|u|^2}\right)^{1-n}dV(\eta)\right)^\frac{p}{2}\om(u)dV(u)\\
&= \int_{\BB\backslash \frac{1}{2}\BB}\left(\int_{\Gamma_{ru}}
|F_{a,p}(\frac{\xi}{r})|^2|\Re g(\xi)|^2\left(1-\frac{|\xi|^2}{|ru|^2}\right)^{1-n}\frac{1}{r^{2n}}dV(\xi)\right)^\frac{p}{2}\om(u)dV(u)\\
&\lesssim\int_{\BB\backslash \frac{1}{2}\BB}\left(\int_{\Gamma_{ru}}|F_{a,p}(\xi)|^2|\Re g(\xi)|^2\left(1-\frac{|\xi|^2}{|ru|^2}\right)^{1-n}dV(\xi)\right)^\frac{p}{2}\om(u)dV(u)  \\
&\approx\int_{r\BB\backslash \frac{r}{2}\BB}\left(\int_{\Gamma_{\zeta}}|F_{a,p}(\xi)|^2|\Re g(\xi)|^2\left(1-\frac{|\xi|^2}{|\zeta|^2}\right)^{1-n}dV(\xi)\right)^\frac{p}{2}
\om(\frac{\zeta}{r})dV(\zeta)   \\
&\approx \int_{\frac{r}{2}}^r \left(\int_{\SS} \left(\int_{\Gamma_{t\eta}}
|F_{a,p}(\xi)|^2|\Re g(\xi)|^2\left(1-\frac{|\xi|^2}{t^2}\right)^{1-n}dV(\xi)\right)^\frac{p}{2}d\sigma(\eta)\right)\om(\frac{t}{r})dt.
\end{align*}
By (\ref{0121-1}) and the similar calculation, when $t\geq \frac{1}{4}$, we have
$$\|(T_gF_{a,p})_t\|_{H^p}^p\approx \int_{\SS} \left(\int_{\Gamma_{t\eta}}
|F_{a,p}(\xi)|^2|\Re g(\xi)|^2\left(1-\frac{|\xi|^2}{t^2}\right)^{1-n}dV(\xi)\right)^\frac{p}{2}d\sigma(\eta).$$
Then
\begin{align}
\|T_{g_r}F_{a,p}\|_{A_\om^p}^p &\lesssim \int_{\frac{r}{2}}^r\|(T_{g}F_{a,p})_t\|_{H^p}^p\om(\frac{t}{r})dt \nonumber\\
&\leq \int_{\frac{r}{2}}^r\|(T_{g}F_{a,p})_{\frac{t}{r}}\|_{H^p}^p\om(\frac{t}{r})dt \nonumber\\
& \leq \|T_gF_{a,p}\|_{A_\om^p}^p   \label{0308-3}\\
&\leq\|T_g\|^P\om(S_a).\nonumber
\end{align}

When $|a|\geq \frac{1}{2-r}$, by (\ref{0112-3}), Lemma \ref{0114-3}, (\ref{0121-2}) and $\gamma$ is large enough, we have
\begin{align*}
\|T_{g_r}F_{a,p}\|_{A_\om^p}^p
&\approx \int_{\BB}\left(\int_{\Gamma_u}|F_{a,p}(\xi)|^2|\Re g_r(\xi)|^2\left(1-\frac{|\xi|^2}{|u|^2}\right)^{1-n}dV(\xi
)\right)^\frac{p}{2}\om(u)dV(u)\\
&\lesssim M^p_\infty(r,\Re g)\int_{\BB}\left(\int_{\Gamma_u}|F_{a,p}(\xi)|^2\left(1-\frac{|\xi|^2}{|u|^2}\right)^{1-n}dV(\xi)
\right)^\frac{p}{2}\om(u)dV(u)\\
&\lesssim \frac{\|T_g\|^p}{(1-|a|^2)^{p}}\int_{\BB}
\left(\int_{\Gamma_u}|F_{a,p}(\xi)|^2\left(1-\frac{|\xi|^2}{|u|^2}\right)^{1-n}dV(\xi)
\right)^\frac{p}{2}\om(u)dV(u).
\end{align*}
Let $\Re^{-1}F_{a,p}(z)=\frac{1}{\frac{\gamma+n}{p}-1}(1-|a|^2)\left(\frac{1-|a|^2}{1-\langle z,a\rangle}\right)^{\frac{\gamma+n}{p}-1}$,
$$J_3=\int_{\BB}
\left(\int_{\Gamma_u\cap \{\xi: |\langle\xi,a \rangle|\geq\frac{1}{4}\}}|F_{a,p}(\xi)|^2\left(1-\frac{|\xi|^2}{|u|^2}\right)^{1-n}dV(\xi)
\right)^\frac{p}{2}\om(u)dV(u),$$
and
$$J_4=\int_{\BB}
\left(\int_{\Gamma_u\cap \{\xi: |\langle\xi,a \rangle|<\frac{1}{4}\}}|F_{a,p}(\xi)|^2\left(1-\frac{|\xi|^2}{|u|^2}\right)^{1-n}dV(\xi)
\right)^\frac{p}{2}\om(u)dV(u).$$
Then, by using (\ref{0112-3}) and  Lemma \ref{1208-4}, we have
\begin{align*}
|\Re(\Re^{-1}F_{a,p})(z)|=|F_{a,p}(z)\langle  z,a\rangle|\approx |F_{a,p}(z)|, \mbox{ when } |\langle z,a\rangle|\geq \frac{1}{4},
\end{align*}
\begin{align*}
J_3\leq \|\Re^{-1}F_{a,p}\|_{A_\om^p}^p \approx (1-|a|)^p \om(S_a),
\end{align*}
\begin{align*}
J_4
&\lesssim (1-|a|)^{n+\gamma}\int_{\BB}
\left(\int_{\Gamma_u}\left(1-\frac{|\xi|^2}{|u|^2}\right)^{1-n}dV(\xi)
\right)^\frac{p}{2}\om(u)dV(u).
\end{align*}
For any $u\in\BB\backslash\{0\}$ and $\xi\in\Gamma_u\backslash\{0\}$, let $r=|\xi|$ and $\eta=\frac{\xi}{|\xi|}$. Then we have
$$
\xi\in\Gamma_u\Leftrightarrow \left|1-\langle \eta, \frac{ru}{|u|^2}\rangle\right|\leq \frac{\alpha}{2}\left(1-\frac{r^2}{|u|^2}\right).
$$
By (\ref{0625-1}), we have
\begin{align*}
\int_{\Gamma_u}\left(1-\frac{|\xi|^2}{|u|^2}\right)^{1-n}dV(\xi)
&=2n\int_0^{|u|}r^{2n-1}\left(1-\frac{|r|^2}{|u|^2}\right)^{1-n}dr\int_{\{\eta\in\SS:\left|1-\langle \eta, \frac{ru}{|u|^2}\rangle\right|\leq \frac{\alpha}{2}\left(1-\frac{r^2}{|u|^2}\right)\}}d\sigma(\eta)\\
&\approx\int_0^{|u|}r^{2n-1}\left(1-\frac{|r|^2}{|u|^2}\right)dr\lesssim 1.
\end{align*}
By Lemma \ref{1210-3}, when $\gamma$ is large enough, we have $J_4\lesssim (1-|a|)^p\om(S_a)$. So, we have
$\|T_{g_r}F_{a,p}\|_{A_\om^p}^p\lesssim\|T_g\|^p\om(S_a)$,
that is, (\ref{0119-6}) holds. The proof  is complete.
\end{proof}

\begin{Theorem}\label{0308-2}
Let $0< p\leq q<\infty$, $\om\in\hD$, $\kappa=\frac{1}{p}-\frac{1}{q}$ and $g\in H(\BB)$.
\begin{enumerate}
  \item [(i)]If $0<n\kappa<1$, then the following conditions are equivalent:
  \begin{enumerate}
    \item [(ia)] $T_g:A_\om^p\to A_\om^q$ is compact;
    \item [(ib)] $M_\infty(r,\Re g)=o\left( \frac{\om^{\kappa}(S_a)}{1-r}\right)$;
    \item [(ic)] $g\in \mathcal{C}_0^{2\kappa+1}(\om^*)$.
  \end{enumerate}
  \item [(ii)] The following conditions are equivalent.
  \begin{enumerate}
    \item [(iia)] $T_g:A_\om^p\to A_\om^p$ is compact;
    \item [(iib)] $g\in \mathcal{C}_0^1(\om^*)$.
  \end{enumerate}
\end{enumerate}
\end{Theorem}

\begin{proof}By Lemmas  \ref{0114-3} and \ref{0114-4},  we have  $(ia)$$\Rightarrow$$(ib)$$\Leftrightarrow$$(ic)$.
Let $d\mu_g(z)=|\Re g(z)|^2\om^*(z)dV(z)$. First, we prove  $(i)$.

Suppose  that $(ic) $ holds and  $q=2$. Then  $\mu_g$ is a vanishing 2-Carleson measure for $A_\om^p$.
Using (\ref{0112-4}), we have
$$\|T_g f\|_{A_\om^2}^2\approx\int_{\BB}|f(z)|^2|\Re g(z)|^2\om^*(z)dV(z).$$
So, $T_g:A_\om^p\to A_\om^2$ is compact by Theorem \ref{0109-7}.

Suppose that $(ic)$ holds.  By Theorem \ref{0121-4}, $T_g:A_\om^p\to A_\om^q$ is bounded.
For every $\varepsilon>0$, there is a $r\in(0,1)$, such that
$$\sup\limits_{|a|\geq r}\frac{\mu_g(S_a)}{(\om(S_a))^{2\kappa+1}}<\varepsilon.$$
For any measurable subset $E$ of $\BB$,  define $\mu_{g,r}(E)=\mu_g(E\cap(\BB\backslash r\BB))$.
By  (\ref{0122-2}),
$$\sup\limits_{a\in\BB}\frac{\mu_{g,r}(S_a)}{(\om(S_a))^{2\kappa+1}}\lesssim\varepsilon.$$

When   $q>2$, let $\beta=\frac{(2\kappa+1)q}{2\kappa q+2}$ and $\beta^\p=\frac{(2\kappa+1)q}{q-2} $.
For any $0<r<1$, by (\ref{0112-4}),  H\"{o}lder's inequality and Theorem \ref{0109-7}, we have
\begin{align*}
\|T_gf\|_{A_\om^q}^q
\approx& \left(\int_{r\BB}+\int_{\BB\backslash r\BB}\right) |f(z)|^2|T_gf(z)|^{q-2}|\Re g(z)|^2\om^*(z)dV(z)   \\
\leq & \left(\int_{r\BB}|f(z)|^{2\beta}d\mu_g(z)\right)^\frac{1}{\beta}
    \left(\int_{r\BB}|T_g f(z)|^{(q-2)\beta^\p}d\mu_g(z)\right)^\frac{1}{\beta^\p}  \\
    &+\left(\int_{\BB\backslash r\BB}|f(z)|^{2\beta}d\mu_g(z)\right)^\frac{1}{\beta}
    \left(\int_{\BB\backslash r\BB}|T_g f(z)|^{(q-2)\beta^\p}d\mu_g(z)\right)^\frac{1}{\beta^\p}  \\
\lesssim & \sup_{|z|<r}|f(z)|^2     \left(\int_{\BB}|T_g f(z)|^{(q-2)\beta^\p}d\mu_g(z)\right)^\frac{1}{\beta^\p}  \\
    &+\left(\int_{\BB}|f(z)|^{2\beta}d\mu_g(z)\right)^\frac{1}{\beta}
    \left(\int_{\BB}|T_g f(z)|^{(q-2)\beta^\p}d\mu_{g,r}(z)\right)^\frac{1}{\beta^\p}  \\
 \lesssim &  \sup_{|z|\leq r}|f(z)|^2  \cdot         \|T_gf\|_{A_{\om}^\frac{(q-2)\beta^\p}{2\kappa+1}}^{q-2}
 +\varepsilon^\frac{1}{\beta^\p}  \|f\|_{A_{\om}^\frac{2\beta}{2\kappa+1}}^2         \|T_gf\|_{A_{\om}^\frac{(q-2)\beta^\p}{2\kappa+1}}^{q-2}  \\
 =&  \sup_{|z|\leq r}|f(z)|^2  \cdot         \|T_gf\|_{A_\om^q}^{q-2}
 +\varepsilon^\frac{1}{\beta^\p}  \|f\|_{A_{\om}^p}^2         \|T_gf\|_{A_{\om}^q}^{q-2}  .
\end{align*}
Then by Lemma \ref{1210-2}, $(ia)$ holds.

When $0<q<2$, by (\ref{0122-3}) and (\ref{0117-5}), we have
\begin{align*}
\|T_g f\|_{A_\om^q}^q
&\lesssim \|f\|_{A_\om^p}^\frac{(2-q)p}{2} \left(\int_{\BB}|f(z)|^{2-\tau}|\Re g(z)|^2\om^*(z)dV(z)\right)^\frac{q}{2}   \\
&\lesssim \|f\|_{A_\om^p}^\frac{(2-q)p}{2} \left(\sup_{|z|<r}|f(z)|^{2-\tau}+\int_{\BB}|f(z)|^{2-\tau}d\mu_{g,r}(z)\right)^\frac{q}{2}  \\
&\lesssim \|f\|_{A_\om^p}^\frac{(2-q)p}{2} \left(\sup_{|z|<r}|f(z)|^{2-\tau}+\varepsilon \|f\|_{A_\om^\frac{2-\tau}{2\kappa+1}}^{2-\tau}\right)^\frac{q}{2}           \\
&= \|f\|_{A_\om^p}^\frac{(2-q)p}{2} \left(\sup_{|z|<r}|f(z)|^{2-\tau}+\varepsilon \|f\|_{A_\om^p}^{2-\tau}\right)^\frac{q}{2}.
\end{align*}
Here $\tau=\frac{(2-q)p}{q}$. By Lemma \ref{1210-2}, $T_g:A_\om^p\to A_\om^q$ is compact.
So, we finish the proof of $(ic)$$\Rightarrow$$(ia)$.

When $p=2$, $(iia)$$\Leftrightarrow$$(iib)$ is obvious.

When $p\neq 2$, by the proof of $(ic)$$\Rightarrow$$(ia)$, we get $(iib)$$\Rightarrow$$(iia)$.

Suppose $p>2$ and $(iia)$ holds. By (\ref{0122-4}), we have
\begin{align*}
\int_{\BB}|f(z)|^p|\Re g(z)|^2\om^*(z)d\nu(z)
\lesssim \|f\|_{A_\om^p}^{p-2}\|T_g f\|_{A_\om^p}^2.
\end{align*}
Let $f_{a,p}(z)=\frac{F_{a,p}(z)}{\|F_{a,p}\|_{A_\om^p}}$ for some $\gamma$ is large enough. Then we have
\begin{align*}
\frac{\mu_g(S_a)}{\om(S_a)}
=\int_{S_a} |f_{a,p}|^pd\mu_g(z)
\leq \int_{\BB}|f_{a,p}(z)|^p|\Re g(z)|^2\om^*(z)d\nu(z)
\lesssim \|T_g f_{a,p}\|_{A_\om^p}^2.
\end{align*}
By Lemma \ref{1210-2}, $(iib)$ holds.

Suppose $0<p<2$ and $(iia)$ holds.
Let $f_{a,p}(z)=\frac{F_{a,p}(z)}{\|F_{a,p}\|_{A_\om^p}}$ for some $\gamma$ is large enough.
Then $\sup_{a\in\BB}\frac{\mu_g(S_a)}{\om(S_a)}<\infty$.
By (\ref{0119-4}), we have
\begin{align*}
\frac{\mu_g(S_a)}{\om(S_a)}\lesssim \|T_g f_{a,p}\|_{A_\om^p}^{p}
\left(\sup_{a\in\BB}\frac{\mu_g(S_a)}{\om(S_a)}\right)^{\frac{\tau_2}{\tau_3\tau_1^\p}}.
\end{align*}
By Lemma \ref{1210-2}, $(iib)$ holds. The proof is complete.
\end{proof}\msk

\section{inclusion relations about $\mathcal{C}^1(\om^*)(\mathcal{C}_0^1(\om^*))$}

 In this section, we discuss   the inclusion relationship between $\mathcal{C}^1(\om^*)(\mathcal{C}_0^1(\om^*))$ and some other function spaces,
such as $\B(\B_0)$ and $BMOA (VMOA)$.

 Recall that
a function $f\in H(\BB)$ is said to belong to the Bloch space, denoted by ${\B}={\mathcal B}(\BB)$, if
  \begr\|f\|_{{\B}}=|f(0)|+ \sup_{z\in \BB}(1-|z|^2)\left|\Re f(z)\right|<\infty.   \nonumber
\endr
 It is well known that ${\mathcal B}$ is a Banach space with the above norm. Let $\mathcal{B}_0$, called the
little Bloch space, denote the subspace of $\mathcal{B}$ consisting
of those $f \in \mathcal{B}$ for which $$\lim_{|z|\to 1}(1-|z|^2)|\Re f(z)|= 0.\nonumber$$
Recall that
$$S^*(\xi,r)=\{z\in\BB:|1-\langle z,\xi\rangle|<r\}, \,\,\mbox{ for }\,\,\xi\in\SS.$$
A function $f \in H(\BB)$ is said to belong to the space $BMOA$ if and only if
  $$\sup_{\xi\in\SS,0<r<\delta}\frac{\int_{S^*(\xi,r)}(1-|z|^2)|\Re f(z)|^2dV(z)}{r^{2n}}<\infty$$
  for some (equivalently, for any) $\delta>0$.
   Let $VMOA$ denote the subspace of $BMOA$ for which
  $$\lim_{r\to 1}\sup_{\xi\in\SS}\frac{\int_{S^*(\xi,r)}(1-|z|^2)|\Re f(z)|^2dV(z)}{r^{2n}}=0.$$
More information about $\B$, BMOA and VMOA can be found in \cite{Zk2005}.

\begin{Proposition}\label{0117-3}
\begin{enumerate}[(i)]
  \item If $\om\in\hD$, then $\mathcal{C}^1(\om^*)\subset A_\om^p$ for all $p>0$.
  \item If $\om\in\hD$, then $BMOA\subset \mathcal{C}^1(\om^*)\subset \B$  and $VMOA\subset \mathcal{C}_0^1(\om^*)\subset \B_0$.
  \item If $\om\in\hD$, with the norm $\|\cdot\|_{\mathcal{C}^1(\om^*)}$, ${C}^1(\om^*)$ is a Banach space and $\mathcal{C}_0^1(\om^*)$ is a closed subspace of ${C}^1(\om^*)$.
  \item If $\om\in \R$, then $\mathcal{C}^1(\om^*)=\B$  and $\mathcal{C}_0^1(\om^*)=\B_0$.
  \item If $\om\in\I$, then $\mathcal{C}^1(\om^*)\subsetneq\B$  and $\mathcal{C}_0^1(\om^*)\subsetneq\B_0$.
  \item If $\om\in \I$ and both  $\om$ and $\frac{\hat{\om}(r)}{(1-r)\om(r)}$ are increasing on $[0,1)$, then $VMOA\subsetneq \mathcal{C}_o^1(\om^*)$ and  $BMOA\subsetneq \mathcal{C}^1(\om^*)$.
\end{enumerate}
\end{Proposition}

\begin{proof} $(i)$. Suppose $g\in\mathcal{C}^1(\om^*)$.
By Theorem \ref{0113-2},
\begin{align}\label{0124-1}
\|g-g(0)\|_{A_\om^2}^2\approx \int_{\BB}|\Re g(z)|^2\om^*(z)dV(z)\lesssim \|h\|_{A_\om^2}^2,\,
,\mbox{here}\,\,h(z) \equiv 1.
\end{align}
So, $g\in A_\om^2$. Similarly, we have
$$
\|g-g(0)\|_{A_\om^4}^4\approx \int_{\BB}|g(z)-g(0)|^2|\Re g(z)|^2\om^*(z)dV(z)\lesssim \|g-g(0)\|_{A_\om^2}^2.
$$
By mathematical induction, for all $k\in\N$, $g-g(0)\in A_{\om}^{2k}$. So, $\mathcal{C}^1(\om^*)\subset A_\om^p$ for all $p>0$.

$(ii)$. By Theorems \ref{0121-4}, \ref{0308-2} and Lemma \ref{0114-3}, we have $\mathcal{C}^1(\om^*)\subset \B$  and $\mathcal{C}_0^1(\om^*)\subset \B_0$.
Suppose $g\in BMOA$, let $d\mu_g^*(z)=(1-|z|^2)|\Re g(z)|^2dV(z).$
Then
$$
M:=\sup\left\{\frac{\mu_g^*(S^*(\xi,r))}{r^{n}}:\xi\in\SS,0<r<\delta\right\}<\infty,
$$
where $\delta$ is any fixed positive constant.
By Proposition \ref{0118-2},
for any $a\in \BB$ with $|a|>\frac{1}{2}$, there are $\xi\in\SS$ and  $r=2(1-|a|)$, such that $S_a\subset S^*(\xi,r)$.
By Lemma \ref{1210-3}, we have
\begin{align*}
\frac{\int_{S_a}|\Re g(z)|^2\om^*(z)dV(z)}{\om(S_a)}
&\lesssim \frac{\hat{\om}(a)\int_{S_a}|\Re g(z)|^2(1-|z|)dV(z)}{(1-|a|)^n\hat{\om}(a)}\\
&\lesssim \frac{\int_{S^*(\xi,r)}|\Re g(z)|^2(1-|z|^2)dV(z)}{r^n}\leq M.
\end{align*}
So, $g\in \mathcal{C}^1(\om^*)$. That is $BMOA\subset \mathcal{C}^1(\om^*)$.
Similarly, by Theorem 5.19 in \cite{Zk2005}, we have $VMOA\subset \mathcal{C}_0^1(\om^*)$.

$(iii)$. Suppose that $\{g_k\}$ is a Cauchy's sequence in $\mathcal{C}^1(\om^*)$.
By (\ref{0124-1}) and Theorem \ref{0109-7}, $\{g_k\}$ is a Cauchy's sequence in $A_\om^2$.
Then we have $g\in A_\om^2$ such that $\lim\limits_{k\to\infty}\|g_k-g\|_{A_\om^2}=0.$
By Theorem \ref{0113-2}, Fatou's Lemma and Theorem \ref{0109-7}, for any $f\in A_\om^2$, we have
\begin{align*}
\|T_gf\|_{A_\om^2}^2
&\approx \int_{\BB} |f(z)|^2|\Re g(z)|^2\om^*(z)dV(z)\\
&=\int_{\BB} |f(z)|^2\liminf_{k\to\infty}|\Re g_k(z)|^2\om^*(z)dV(z)  \\
&\leq \liminf_{k\to\infty}\int_{\BB} |f(z)|^2|\Re g_k(z)|^2\om^*(z)dV(z)\\
&\lesssim \liminf_{k\to\infty} \|g_k\|_{\mathcal{C}^1(\om^*)}\|f\|_{A_\om^2}^2.
\end{align*}
So, $T_g:A_\om^2\to A_\om^2$ is bounded. Then $g\in \mathcal{C}^1(\om^*)$. Similarly, for all $f\in A_\om^2$, we have
\begin{align*}
\int_{\BB}|f(z)|^2|(\Re g-\Re g_j)(z)|^2\om^*(z)dV(z)\approx &\|(T_g-T_{g_j})f\|_{A_\om^2}^2  \\
\lesssim & \liminf_{k\to\infty} \|g_j-g_k\|_{\mathcal{C}^1(\om^*)}\|f\|_{A_\om^2}^2.
\end{align*}
By Theorem \ref{0109-7}, $\lim\limits_{j\to\infty}\|g-g_j\|_{\mathcal{C}^1(\om^*)}=0$.
So, $\mathcal{C}^1(\om^*)$ is a Banach space.

Suppose $\{g_k\}$ is a Cauchy's sequence in $\mathcal{C}_0^1(\om^*)$. Then there exists $g\in\mathcal{C}^1(\om^*)$ such that
$\lim_{k\to\infty}\|g_k-g\|_{\mathcal{C}^1(\om^*)}=0$.
Let $\{f_j\}$ be a bounded sequence in $A_\om^2$ such that $\{f_j\}$ converges to 0 uniformly on compact subsets of $\BB$.
By Theorems \ref{0109-7} and \ref{0113-2}, we have
\begin{align*}
\|T_g f_j\|_{A_\om^2}
\leq \|T_{g-g_k}f_j\|_{A_\om^2}+\|T_{g_k}f_j\|_{A_\om^2}
\lesssim \|g-g_k\|_{\mathcal{C}^1(\om^*)}^\frac{1}{2}\|f_j\|_{A_\om^2}+\|T_{g_k}f_j\|_{A_\om^2}.
\end{align*}
For any given $\varepsilon>0$, we can choose a $k\in\N$ such that $ \|g-g_k\|_{\mathcal{C}^1(\om^*)}<\varepsilon^2$.
By Lemma \ref{1210-2},
$$\lim_{j\to\infty}\|T_{g_k} f_j\|_{A_\om^2}
\lesssim \varepsilon \sup_{j\geq 1}\left\{\|f_j\|_{A_\om^2}\right\}.$$ Then $T_g:A_\om^2\to A_\om^2$ is compact. So, $g\in \mathcal{C}_0^1(\om^*)$.
That is,  $\mathcal{C}_0^1(\om^*)$ is a closed subspace of ${C}^1(\om^*)$.

$(iv)$. Suppose $\om\in \R$. By observation (v) after Lemma 1.1 in \cite{PjaRj2014book},
there exists $\beta>-1$ and $\delta\in (0,1)$, such that
$\frac{\om(r)}{(1-r)^\beta}$ is decreasing on $[\delta,1)$. Without loss of generality, let $\delta=0$.
Then for all $g\in\B$ and $a\in\BB$ such that $|a|>\frac{1}{2}$, by Lemmas \ref{1212-1} and \ref{1210-3}, we have
\begin{align*}
\int_{S_a}|\Re g(z)|^2\om^*(z)dV(z)
&\approx \int_{S_a}|\Re g(z)|^2(1-|z|)^{2+\beta}\frac{\om(a)}{(1-|z|)^\beta}dV(z)   \\
&\leq \frac{\|g\|_{\B}^2  \om(a) }{(1-|a|)^\beta} \int_{S_a}(1-|z|)^{\beta}dV(z)  \\
&\approx  \frac{\|g\|_{\B}^2  \om(a) }{(1-|a|)^\beta} \int_{S_a}(1-|z|)^{\beta}dV(z)\\
&\approx \|f\|_{\B}\om(S_a).
\end{align*}
Then $\B\subset \mathcal{C}^1(\om^*)$. So, $\B= \mathcal{C}^1(\om^*)$. Similarly, $\B_0= \mathcal{C}_0^1(\om^*)$.

$(v)$ and $(vi)$ have been proved in \cite{PjaRj2014book} when $n=1$, so they also hold for $n>1$. The proof is complete.
\end{proof}

\begin{Proposition}
Let $\om\in\hD$ and $g\in\mathcal{C}^1(\om^*)$. Then the following statements are equivalent.
\begin{enumerate}[(i)]
  \item $g\in C_0^1(\om^*)$;
  \item $\lim\limits_{r\to 1}\|g-g_r\|_{\mathcal{C}^1(\om^*)}=0$, here $g_r(z)=g(rz)$;
  \item There is a sequence of polynomials $\{p_k\}$ such that $\lim\limits_{r\to 1}\|g-p_k\|_{\mathcal{C}^1(\om^*)}=0$.
\end{enumerate}

\begin{proof}$(i)$$\Rightarrow$$(ii)$. Suppose $g\in C_0^1(\om^*)$.
Let $\gamma$ be large enough and
$$f_{a,2}(z)=\frac{F_{a,2}(z)}{(\om(S_a))^\frac{1}{2}}.$$
Then Lemma \ref{1210-2}, Theorems \ref{0109-7} and \ref{0113-2} yield
\begin{align}\label{0124-3}
\lim_{|a|\to 1}\|T_g(f_{a,2})\|_{A_\om^2}^2\approx \lim_{|a|\to 1} \int_{\BB}|f_{a,2}(z)|^2|\Re g(z)|^2\om^*(z)dV(z)=0
\end{align}
and for any $r\in(0,1)$,
\begin{align}\label{0308-7}
\frac{\int_{S_a}|(\Re g-\Re g_r)(z)|^2\om^*(z)dV(z)}{\om(S_a)}\lesssim  \|T_{g-g_r}f_{a,2}\|_{A_\om^2}^2.
\end{align}
By (\ref{0124-3}) and $(ii)$  in Proposition \ref{0117-3}, for any $\varepsilon>0$, there is a $r_0\in (\frac{1}{2},1)$ such that
\begin{align}\label{0308-4}
\|T_g(f_{a,2})\|_{A_\om^2}<\varepsilon,\,\mbox{  and }\, (1-|a|)|\Re g(a)|<\varepsilon, \,\mbox{ when }\,|a|>r_0.
\end{align}
By  (\ref{0308-3}), if  $r_0<|a|< \frac{1}{2-r}$, we have
\begin{align}\label{0308-5}
\|T_{g-g_r}f_{a,2}\|_{A_\om^2}^2
\lesssim \|T_{g}f_{a,2}\|_{A_\om^2}^2  +\|T_{g_r}f_{a,2}\|_{A_\om^2}^2
\lesssim  \|T_g f_{a,2}\|_{A_\om^2}^2\leq \varepsilon^2.
\end{align}
If $|a|>\max\left\{r_0, \frac{1}{2-r}\right\}$, by (\ref{0112-4}), we have
\begin{align*}
\|T_{g-g_r}f_{a,2}\|_{A_\om^2}^2
&\lesssim \|T_{g}f_{a,2}\|_{A_\om^2}^2  +\|T_{g_r}f_{a,2}\|_{A_\om^2}^2  \\
&\leq \varepsilon^2 + \int_{\BB}|\Re g_r(z)|^2|f_{a,p}(z)|^2\om^*(z)dV(z)  \\
&\leq  \varepsilon^2 + M_\infty^2(r,\Re g)\int_{\BB}|f_{a,p}(z)|^2\om^*(z)dV(z).
\end{align*}
By  Theorem 1.12 in \cite{Zk2005} and Lemma \ref{1210-3}, if $\gamma$ is large enough,  we have
\begin{align*}
\int_{\BB}|f_{a,p}(z)|^2\om^*(z)dV(z)
&=\frac{(1-|a|^2)^{\gamma+n}}{\om(S_a)}\int_\BB \frac{\om^*(z)}{|1-\langle z,a\rangle|^{\gamma+n}}dV(z)  \\
&\lesssim \frac{(1-|a|^2)^{\gamma+n}\om^*(a)}{(1-|a|^2)^{\gamma-1}\om(S_a)}
\approx {(1-|a|^2)^2}.
\end{align*}
So, when  $|a|>\max\left\{r_0, \frac{1}{2-r}\right\}$, by (\ref{0308-4}),
\begin{align}\label{0308-6}
\|T_{g-g_r}f_{a,2}\|_{A_\om^2}^2
\lesssim \varepsilon^2 +(1-|a|^2)^2 M_\infty^2(2-\frac{1}{|a|},\Re g)
\lesssim \varepsilon^2.
\end{align}
By (\ref{0308-7}), (\ref{0308-5}) and (\ref{0308-6}), we obtain
$$
\sup_{|a|>r_0}\frac{\int_{S_a}|(\Re g-\Re g_r)(z)|^2\om^*(z)dV(z)}{\om(S_a)}\lesssim \varepsilon^2.
$$
When $|a|\leq r_0$, we have
\begin{align*}
\frac{\int_{S_a}|(\Re g-\Re g_r)(z)|^2\om^*(z)dV(z)}{\om(S_a)}
\leq&  \frac{\|\Re g-\Re g_r\|_{A_{\om^*}^2}^2}{\om(S_{r_0})}.
\end{align*}
So, there is a $r_1\geq r_0$, such that
$$\sup_{r>r_1}\frac{\|\Re g-\Re g_r\|_{A_{\om^*}^2}^2}{\om(S_{r_0})}\leq \varepsilon^2.$$
Therefore,
$$
\sup_{a\in\BB, r>r_1}\frac{\int_{S_a}|(\Re g-\Re g_r)(z)|^2\om^*(z)dV(z)}{\om(S_a)}\lesssim \varepsilon^2.
$$
Then $(ii)$ holds.

$(ii)$$\Rightarrow$$(iii)$. For any $n\in\N$, there is a polynomial $p_n$ such that $$\|\Re g_{1-\frac{1}{n}}-\Re p_n\|_{H^\infty}<\frac{1}{n}.  $$
 Since
\begin{align*}
\|g-p_n\|_{\mathcal{C}^1(\om^*)}\leq& \|g-g_{1-\frac{1}{n}}\|_{\mathcal{C}^1(\om^*)} +\|g_{1-\frac{1}{n}}-p_n\|_{\mathcal{C}^1(\om^*)} \\
\lesssim&  \|g-g_{1-\frac{1}{n}}\|_{\mathcal{C}^1(\om^*)} +\|\Re g_{1-\frac{1}{n}}-\Re p_n\|_{H^\infty}^2,
\end{align*}
we obtain $(iii)$.

$(iii)$$\Rightarrow$$(i)$. For any polynomial $p_n$, we have $\|\Re p_n\|_{H^\infty}<\infty$.
 Then by Lemmas \ref{1212-1} and \ref{1210-3}, for $|a|>\frac{1}{2}$ we have
\begin{align*}
\frac{\int_{S_a}|\Re p_n(z)|^2\om^*(z)dV(z)}{\om(S_a)} &\lesssim \|\Re p_n\|_{H^\infty}^2 \frac{(1-|a|)\hat{\om}(a)\int_{S_a}dV(z)}{\om(S_a)}\\
&\approx (1-|a|)^2 \|\Re p_n\|_{H^\infty}^2.
\end{align*}
So, $p_n\in \mathcal{C}_0^1(\om^*)$. Then by $(iii)$ of  Proposition \ref{0117-3}, $(i)$ holds.
The proof is complete.
\end{proof}
\end{Proposition}

\end{document}